\documentclass[12pt,reqno]{amsart}
\usepackage[headings]{fullpage}
\usepackage{amssymb,amsmath,mathtools,bbm,tikz,tikz-cd,color,mathrsfs}
\usepackage[bookmarks=true,%
    colorlinks=true,%
    linkcolor=blue,%
    citecolor=blue,%
    filecolor=blue,%
    menucolor=blue,%
    urlcolor=blue,%
    breaklinks=true]{hyperref}

\newtheorem{theorem}{Theorem}[section]
\theoremstyle{definition}
\newtheorem{proposition}[theorem]{Proposition}
\newtheorem{lemma}[theorem]{Lemma}
\newtheorem{definition}[theorem]{Definition}
\newtheorem{remark}[theorem]{Remark}
\newtheorem{corollary}[theorem]{Corollary}

\newtheorem{example}[theorem]{Example}

\def\a{\alpha}
\def\b{\beta}

\def\d{\delta}

\def\ve{\varepsilon}
\def\th{\theta}

\def\BZ{\mathbbm Z}
\def\BQ{\mathbbm Q}

\def\BK{\mathbbm K}

\def\BT{\mathbbm T}

\def\CN{\mathcal N}

\def\id{\mathrm{id}}
\def\End{\mathrm{End}}

\def\LG{\mathrm{LG}}

\def\1{\boldsymbol{1}}

\renewcommand{\phi}{\varphi}

\def\BE{\mathbbm E}   
\def\BB{\mathbbm B}
\def\CA{\mathcal A}
\def\CB{\mathcal B}    
\def\CC{\mathcal C}
\def\CE{\mathcal E}          
\def\CF{\mathcal F}

\def\CS{\mathcal S}

\def\ext{{\textstyle\bigwedge}}   \def\T{\mathsf{T}}             
\def\tB{\widetilde{B}}
\def\tP{\widetilde{P}}
\def\Rp{\mathsf R}                  
\def\Rm{\mathsf R^{-}}            
\def\bx{\bar x}
\def\by{\bar y}
\def\Ga{G}                        
\def\Gm{G^{-}}
\def\hate{\hat e}                   
\def\cdag{c^{\dagger}}
\def\tr{\mathrm{tr}}

\def\diag{\mathrm{diag}}
\def\sgn{\mathrm{sgn}}
\def\spn{\mathrm{span}}

\begin{document}

\title[]{The Links--Gould polynomial of an alternating link is alternating}

\author[M. Harper]{Matthew Harper}
\address{Department of Mathematics, Michigan State University, East Lansing, Michigan, USA}
\email{mrhmath@proton.me}
   
\author[B.-M. Kohli]{Ben-Michael Kohli}
\address{Section de Math\'ematiques, Universit\'e de Gen\`eve \\
rue du Conseil-G\'en\'eral 7-9, 1205 Gen\`eve, Switzerland}
\email{bm.kohli@protonmail.ch}

\thanks{
  {\em Key words and phrases:}
  Links--Gould invariant, Alexander polynomial, alternating links, Burau representation.
}

\date{\today}

\begin{abstract}
Building on work of Kauffman--Saleur and earlier work of the second author, we derive a determinant formula for the Links--Gould polynomial of links. As an application, we prove a 20-year-old conjecture of Ishii: for any alternating link, the Links--Gould invariant is an alternating polynomial.
\end{abstract}
\maketitle

{\footnotesize
\tableofcontents
}

\section{Introduction}

The Links--Gould invariant of links $\LG^{2,1}(L;t_{0},t_{1}) = \LG(L;t_{0},t_{1})  \in \mathbb{Z}[t_0^{\pm 1} , t_1^{\pm 1}]$ is obtained from the Reshetikhin--Turaev construction applied to representations of the Hopf superalgebras $U_{q}\mathfrak{gl}(2 \vert 1)$, see \cite{LinksGould, DeWitKauffmanLinks, DW01}. The $\LG$ invariant is closely related to the Alexander polynomial $\Delta_L$: recent results show that it reproduces some of the Alexander polynomial's most distinctive characteristics. Specifically, the degree of the Links--Gould polynomial gives a lower bound for the genus of a knot by work of two of the authors \cite{Kohli-Tahar}. Also, the $\LG$ invariant of a fibered link is $\mathbb{Z}[q^{\pm 1}]$-monic following the recent work of L\'opez-Neumann and van der Veen \cite{LNVdV25}. These two properties imply and improve upon similar statements for the Alexander polynomial. This is because $\LG(L;t_0,t_1)$ is known to specialize to $\Delta_L(t)$ in two different ways \cite{DWIL, Kohli, KPM}: 
$$
\LG(L;t_0,-t_0^{-1}) = \Delta_L(t_0^2)\,, \qquad
\LG(L; t_0, t_0^{-1}) = \Delta_L(t_0)^2.
$$
It is therefore reasonable to imagine that other properties of the Alexander polynomial should have their counterparts for the Links--Gould invariant, and that $\LG$ has an expression that makes it more closely related to the classical homological definitions of the Alexander polynomial.

For an alternating knot $K$, Ishii conjectured that the coefficients $a_{ij}$ of
$LG_K(t_0,t_1)=\sum a_{ij}t_0^it_1^j$ satisfy $a_{ij}a_{i'j'}\ge0$ if $i+j-i'-j'$ is even and $a_{ij}a_{i'j'}\le0$ otherwise, and
he verified this for the prime knots with at most $10$ crossings \cite[Section~5]{Ishii}; see also
\cite[Conjecture~1.1]{HKST}. For an alternating link $L$ with $\mu$ components, the two authors, Song and Tahar
conjectured that $a_{ij}=(-1)^{\mu+i+j+1}|a_{ij}|$ for all $i,j$ and that the array of the coefficients has no interior zeros \cite[Conjecture~1.3]{HKST}; for knots, the first of these two statements is also part of \cite[Conjecture~1.1]{GLY}. Here we prove the alternating conjecture for the Links--Gould polynomial on links.

\begin{theorem}\label{thm.alternating}
Let $L$ be an alternating link with $\mu$ components. Then the Links--Gould polynomial of $L$ is alternating, namely,
\[
\LG(L;t_0,t_1)=\sum_{ij}(-1)^{\mu-1+i+j}|a_{ij}| t_0^it_1^j.
\]
Moreover, $\LG(L)\ne0$ if $L$ has a connected
alternating diagram.
\end{theorem}

The other statements conjectured in \cite{HKST} and \cite{GLY} (absence of interior
zeros, log-concavity, unimodality) are not addressed here. 

Our proof of Theorem \ref{thm.alternating} uses the $V_1$-polynomial of Garoufalidis and Kashaev \cite{GaroufalidisKashaev}, which is equal to Links--Gould: $$V_1(L;t_0,t_1)=\LG(L;t_0,t_1)$$ 
for every link $L$, in the
variables of Ishii. This fact is proved in \cite{GHKKST}.

Crucially, we find an expression for the $\LG = V_1$ polynomial that extends both the work of Kauffman--Saleur \cite{KauffmanSaleur} and earlier work of the second author \cite{Kohli}. Kauffman and Saleur showed that the $R$-matrix of the Alexander--Conway polynomial is the exterior power of a Burau matrix, so that the state sum is a Gaussian Berezin integral, i.e.\ a determinant. We carry out the same program for the $R$-matrix of the $V_1$-polynomial of Garoufalidis--Kashaev, in the form given in \cite{GHKST}. The space $W$ is the Fock space of two fermions per strand, and for $r=-1$, in a suitable gauge,
\[
   -R\;=\;\ext\bigl(\rho(t_0)\oplus\rho(t_1)^{\T}\bigr)\;+\;(1-t_0t_1)\,\bigl(n_{a,1}-a_1^\dagger a_2\bigr)\bigl(n_{b,1}-b_2^\dagger b_1\bigr),
\]
that is, two copies of the Kauffman--Saleur free-fermion $R$-matrix, at $t_0$ and at $t_1$, plus a single quartic
term. Equivalently, $-R$ is obtained from the free two-flavor $R$-matrix
$\ext(\rho(x)\oplus\rho(y)^{\T})$ by the linear substitution $x\mapsto t_0$, $y\mapsto t_1$, $xy\mapsto 1$. For a braid $\b$, let $B(x)$ and $\tB(x)$ be two flavors of a product of Burau elementary matrices where an independent variable $x=\{x_j\}$ is assigned at each crossing.
\begin{theorem}
    For the closure $\widehat\b$ of a braid $\b$ with crossings $j=1,\dots,c$ of signs $\ve_j$ with $\mu$ components,
\begin{gather*}
   \LG(\widehat\b;t_0,t_1)\;=\;(-1)^{\mu-1}\;\BE\Bigl[\det\bigl(1-B(x)'\bigr)\,\det\bigl(1-\tB(y)'\bigr)\Bigr],\\
   \BE[x_j]=t_0^{\ve_j},\qquad \BE[y_j]=t_1^{\ve_j},\qquad \BE[x_jy_j]=1,
\end{gather*}
where $M'$ denotes $M$ with its first row and column removed.
\end{theorem}

We also give the corresponding Berezin integral and a state sum over pairs of spanning trees of the diagram. We then show that the formula holds for every connected diagram of the link, cut along any of its edges.

All this combined allows us to deduce that the $\LG$-polynomial of an alternating link is alternating, following an argument of Crowell \cite{Crowell} for the Alexander polynomial.

\subsection*{Conventions}
Matrices act on column vectors: the $q$-th column of a matrix is the image of the $q$-th basis vector.
For a matrix $M$ and index sets $T,S$ we write $M[T|S]$ for the submatrix with rows $T$ and columns $S$.
$M^{\T}$ is the transpose. For an $n\times n$ matrix $M$, $M'$ denotes $M$ with its first row and first column
removed. Throughout, $\BK=\BQ(t_0,t_1)$; all quantities that occur are in fact Laurent polynomials with
integer coefficients.

\subsection*{Structure of the Paper}
In Section \ref{sec.setup} we recall the $V_1$ polynomial via the Reshetikhin--Turaev construction with an explicit $R$-matrix. Section \ref{sec.exterior} discusses the framework for the computations in the exterior algebra. Section \ref{sec.rmatrix} describes the representation isomorphism from the $R$-matrix representation of the braid group to Fock space. All this leads to the determinant formula, our main theorem, in Section \ref{sec.theorem}. In Section \ref{sec.lagrangian}, we discuss the Berezinian and matrix-tree-theorem formulation of the determinant formula. Section \ref{sec.diagrams} is dedicated to proving the determinant formula for arbitrary diagrams. Section \ref{sec.alternating} provides the proof of the alternating coefficient conjecture.

\subsection*{Acknowledgements}
BK wishes to thank Guillaume Tahar for the discussions they had a few years ago on the Kauffman-Saleur model, which were a source of inspiration for this work, and Emmanuel Wagner for pointing this model out many years ago. The authors thank Micah Chrisman, David Cimasoni, and Livio Ferretti for interesting conversations. MH was partially supported through the NSF-RTG grant
\#DMS-2135960.

\section{The \texorpdfstring{$R$}{R}-matrix and the invariant}\label{sec.setup}

\subsection{The \texorpdfstring{$R$}{R}-matrix}
Let $W$ be the $\BK$-vector space with basis $e_1,e_2,e_3,e_4$, and write $e_{ij}=e_i\otimes e_j$.
We use two gradings $\a,\b$ on $W$ (called $\deg_{e_2}$, $\deg_{e_3}$ in \cite{GHKST}):
\[
 \a(e_1)=0,\ \a(e_2)=1,\ \a(e_3)=0,\ \a(e_4)=1,\qquad
 \b(e_1)=0,\ \b(e_2)=0,\ \b(e_3)=1,\ \b(e_4)=1 ,
\]
extended additively to tensor products. Let $r\in\BK^\times$. The $R$-matrix $R=R_r\in\End(W\otimes W)$ of
\cite[Appendix~A.1]{GHKST} is given by
\begin{equation}\label{eq.R}
\begin{aligned}
R\,e_{11}&=-e_{11}, & R\,e_{44}&=-e_{44},\\
R\,e_{12}&=-t_0\,e_{21}, & R\,e_{21}&=-e_{12}+(t_0-1)\,e_{21},\\
R\,e_{13}&=-t_1\,e_{31}, & R\,e_{31}&=-e_{13}+(t_1-1)\,e_{31},\\
R\,e_{22}&=t_0\,e_{22}, & R\,e_{33}&=t_1\,e_{33},\\
R\,e_{24}&=r\,t_0t_1\,e_{42}, & R\,e_{42}&=r^{-1}t_1^{-1}\,e_{24}+(t_0-1)\,e_{42},\\
R\,e_{34}&=r^{-1}\,e_{43}, & R\,e_{43}&=r\,t_1\,e_{34}+(t_1-1)\,e_{43},\\
R\,e_{14}&=-t_0t_1\,e_{41}, & R\,e_{23}&=-r\,t_1\,e_{32}+(t_0-1)t_1\,e_{41},
\end{aligned}
\end{equation}
\begin{equation}\label{eq.R2}
\begin{aligned}
R\,e_{32}&=-r^{-1}t_1^{-1}\,e_{23}+r^{-1}(1-t_0)\,e_{41},\\
R\,e_{41}&=-e_{14}+(t_1^{-1}-1)\,e_{23}+r\,(t_1-1)\,e_{32}+(t_0+t_1-2)\,e_{41}.
\end{aligned}
\end{equation}

\begin{lemma}\label{lem.blocks}
$R_r$ preserves the bigrading $(\a,\b)$ of $W\otimes W$, and it is invertible.
\end{lemma}
\begin{proof}
The first claim is read off from \eqref{eq.R}--\eqref{eq.R2}. Hence $R_r$ is block diagonal with respect to the
nine homogeneous components of $W\otimes W$, which have bases
\begin{equation}\label{eq.sectors}
\begin{gathered}
(e_{11});\quad (e_{21},e_{12});\quad (e_{31},e_{13});\quad (e_{22});\quad (e_{33});\quad
(e_{41},e_{23},e_{32},e_{14});\\ (e_{42},e_{24});\quad (e_{43},e_{34});\quad (e_{44}),
\end{gathered}
\end{equation}
of bidegrees $(0,0),(1,0),(0,1),(2,0),(0,2),(1,1),(2,1),(1,2),(2,2)$. The determinants of the nine blocks are
$-1$, $-t_0$, $-t_1$, $t_0$, $t_1$, $t_0t_1$, $-t_0$, $-t_1$, $-1$. 
\end{proof}

\subsection{Braid words and the invariant}
Let $h=\diag(-1,1,1,-1)\in\End(W)$, i.e.\ $h(e_i)=h_ie_i$ with $h_1=h_4=-1$, $h_2=h_3=1$. Note that
\begin{equation}\label{eq.hparity}
   h_i=-(-1)^{\a(e_i)+\b(e_i)} .
\end{equation}
A \emph{braid word} on $n$ strands is a finite sequence $w=s_1s_2\cdots s_c$ of letters
$s_j=\sigma_{\nu_j}^{\ve_j}$ with $1\le \nu_j\le n-1$ and $\ve_j\in\{\pm1\}$. We call $j=1,\dots,c$ the
\emph{crossings} of $w$. For a letter we put
\[
  \varrho_r(\sigma_i^{\pm1})=\id_W^{\otimes(i-1)}\otimes R_r^{\pm1}\otimes\id_W^{\otimes(n-i-1)}\in\End(W^{\otimes n}),\qquad
  \varrho_r(w)=\varrho_r(s_c)\cdots\varrho_r(s_2)\,\varrho_r(s_1),
\]
so that the first letter acts first. The \emph{closure} $\widehat w$ of $w$ is the oriented link obtained by
closing the braid diagram in which the crossings $s_1,\dots,s_c$ are met in this order along the orientation of the
strands, a letter $\sigma_i$ being a positive crossing. 

For a multi-index $I=(m_1,\dots,m_n)\in\{1,2,3,4\}^n$ let $e_I=e_{m_1}\otimes\dots\otimes e_{m_n}$, and for
$v\in W^{\otimes n}$ let $\langle e_I,v\rangle$ be the coefficient of $e_I$ in $v$.

\begin{definition}\label{def.V}
For a braid word $w$ on $n$ strands,
\[
   V(w;r)\;=\;\sum_{(m_2,\dots,m_n)\in\{1,2,3,4\}^{n-1}} h_{m_2}h_{m_3}\cdots h_{m_n}\;
   \bigl\langle e_{(1,m_2,\dots,m_n)},\ \varrho_r(w)\,e_{(1,m_2,\dots,m_n)}\bigr\rangle\ \in\BK .
\]
\end{definition}

In the notation of \cite[Remark~2.3]{GHKST}, $V(w;r)$ is the $(e_1,e_1)$ matrix entry of the operator
$\tr_{2,\dots,n}\bigl((\id_W\otimes h^{\otimes(n-1)})\circ\varrho_r(w)\bigr)\in\End(W)$. 

\begin{theorem}[{\cite[Lemmas~3.1, 3.2 and Remark~2.3]{GHKST}}]\label{thm.GHKST}
For $r=\pm1$ the operator $\tr_{2,\dots,n}\bigl((\id_W\otimes h^{\otimes(n-1)})\circ\varrho_r(w)\bigr)$ is a scalar multiple
of $\id_W$, and the scalar depends only on the oriented link $L=\widehat w$, the closure of $w$.
It is the $V_1$-polynomial $V_1(L)=V_1(L;t_0,t_1)$. In particular $V_1(\widehat w)=V(w;\pm1)$.
\end{theorem}

\subsection{The parameter \texorpdfstring{$r$}{r} is a twist}

\begin{lemma}\label{lem.twist}
Let $F_r\in\End(W\otimes W)$ be the diagonal operator $F_r\,e_{ij}=r^{\b(e_i)\a(e_j)}\,e_{ij}$. Then
$R_r=F_rR_1F_r^{-1}$, and $V(w;r)=V(w;1)$ for every braid word $w$ and every $r\in\BK^\times$.
\end{lemma}
\begin{proof}
Put $\phi(i,j)=\b(e_i)\a(e_j)$; thus $\phi(3,2)=\phi(3,4)=\phi(4,2)=\phi(4,4)=1$ and $\phi(i,j)=0$
otherwise. The identity $R_r=F_rR_1F_r^{-1}$ says that the coefficient of $e_{kl}$ in $R_re_{ij}$ is
$r^{\phi(k,l)-\phi(i,j)}$ times its value at $r=1$. In \eqref{eq.R}--\eqref{eq.R2} the parameter $r$ occurs in
exactly eight coefficients, and one checks them one by one:
\[
\begin{array}{ll}
e_{23}\mapsto e_{32}:\ \ r^{\phi(3,2)-\phi(2,3)}=r, &\qquad e_{32}\mapsto e_{23}:\ \ r^{-1},\\
e_{32}\mapsto e_{41}:\ \ r^{\phi(4,1)-\phi(3,2)}=r^{-1}, &\qquad e_{41}\mapsto e_{32}:\ \ r,\\
e_{24}\mapsto e_{42}:\ \ r^{\phi(4,2)-\phi(2,4)}=r, &\qquad e_{42}\mapsto e_{24}:\ \ r^{-1},\\
e_{34}\mapsto e_{43}:\ \ r^{\phi(4,3)-\phi(3,4)}=r^{-1}, &\qquad e_{43}\mapsto e_{34}:\ \ r,
\end{array}
\]
while for all other non-zero coefficients $\phi(k,l)=\phi(i,j)$ (for instance $\phi(4,1)=\phi(2,3)=\phi(1,4)=0$).

Now let $F^{(n)}\in\End(W^{\otimes n})$ be the diagonal operator $F^{(n)}e_I=r^{\phi_n(I)}e_I$ with
$\phi_n(I)=\sum_{p<q}\b(e_{m_p})\a(e_{m_q})$ for $I=(m_1,\dots,m_n)$. Fix $k$ and split
\begin{align*}
 \phi_n(I)={}&\b(e_{m_k})\a(e_{m_{k+1}})
   +\sum_{p<k}\b(e_{m_p})\bigl(\a(e_{m_k})+\a(e_{m_{k+1}})\bigr)\\
   &+\sum_{q>k+1}\bigl(\b(e_{m_k})+\b(e_{m_{k+1}})\bigr)\a(e_{m_q})+(\text{terms without }m_k,m_{k+1}).
\end{align*}
The operator $\varrho_1(\sigma_k)$ changes only the $k$-th and $(k+1)$-st tensor factors and preserves their total
bidegree (Lemma~\ref{lem.blocks}); therefore only the first term changes, and
$F^{(n)}\varrho_1(\sigma_k^{\pm1})(F^{(n)})^{-1}=\id\otimes F_rR_1^{\pm1}F_r^{-1}\otimes\id=\varrho_r(\sigma_k^{\pm1})$.
Consequently $\varrho_r(w)=F^{(n)}\varrho_1(w)(F^{(n)})^{-1}$. Conjugation by a diagonal operator does not change
diagonal matrix entries, and $V(w;r)$ is a combination of diagonal entries of $\varrho_r(w)$.
\end{proof}

\begin{remark}[Convention]\label{rem.convention}
From now on $r=-1$, and we write $R=R_{-1}$, $\varrho=\varrho_{-1}$ and $V(w)=V(w;-1)$. By
Theorem~\ref{thm.GHKST}, $V(w)=V_1(\widehat w)$: the state sum $V(w)$ of a braid word is the $V_1$-polynomial of its
closure. We keep the notation $V(w)$ for the state sum itself, as given by Definition~\ref{def.V}.
\end{remark}

\section{Computations in the exterior algebra}\label{sec.exterior}

\subsection{Exterior powers of a matrix}

Let $u_1,\dots,u_N$ be an ordered basis of a $\BK$-vector space $U$, and let $\ext(U)$ be the exterior algebra
of $U$.
For a subset $S=\{s_1<\dots<s_k\}$ of $\{1,\dots,N\}$ put
\[
   u_S=u_{s_1}\wedge\dots\wedge u_{s_k},\qquad u_\emptyset=1 .
\]
The $u_S$ form a basis of $\ext(U)$; we write $\langle u_T,\omega\rangle$ for the coefficient of
$u_T$ in $\omega$. Let $M\in\End(U)$, with matrix $(M_{pq})$ defined by $Mu_q=\sum_pM_{pq}u_p$. We denote by
$\ext(M)$ the unique algebra endomorphism of $\ext(U)$ extending $M$:
\[
   \ext(M)(v_1\wedge\dots\wedge v_k)=Mv_1\wedge\dots\wedge Mv_k ,\qquad \ext(M)1=1 .
\]

\begin{lemma}\label{lem.minors}
Let $M,N\in\End(U)$.
\begin{enumerate}\renewcommand{\theenumi}{\roman{enumi}}
\item $\ext(MN)=\ext(M)\ext(N)$ and $\ext(1)=1$.
\item $\langle u_T,\ext(M)u_S\rangle=\det M[T|S]$ if $|T|=|S|$, and $0$ otherwise, where the determinant of the empty
      matrix is $1$.
\end{enumerate}
\end{lemma}
\begin{proof}
(i) Both sides are algebra endomorphisms that agree on $U$. \\
(ii) Let $S=\{s_1<\dots<s_k\}$. Expanding,
$\ext(M)u_S=\sum_{p_1,\dots,p_k}M_{p_1s_1}\cdots M_{p_ks_k}\,u_{p_1}\wedge\dots\wedge u_{p_k}$.
A term vanishes unless the $p_m$ are distinct, and then $\{p_1,\dots,p_k\}=T=\{t_1<\dots<t_k\}$ with
$p_m=t_{\pi(m)}$ for a permutation $\pi$, and $u_{p_1}\wedge\dots\wedge u_{p_k}=\sgn(\pi)\,u_T$. The coefficient of
$u_T$ is therefore $\sum_\pi\sgn(\pi)\prod_mM_{t_{\pi(m)}s_m}=\det M[T|S]$.
\end{proof}

\begin{lemma}[the supertrace is a determinant, {\cite[Appendix]{KauffmanSaleur}}]\label{lem.str}
For $M\in\End(U)$,
\[
   \sum_{S\subseteq\{1,\dots,N\}}(-1)^{|S|}\,\bigl\langle u_S,\ext(M)u_S\bigr\rangle\;=\;\sum_{S}(-1)^{|S|}\det M[S|S]\;=\;\det(1-M).
\]
\end{lemma}
\begin{proof}
The first equality is Lemma~\ref{lem.minors} (ii). The $q$-th column of $1-M$ is $u_q-Mu_q$. By multilinearity of
the determinant in the columns, $\det(1-M)=\sum_S(-1)^{|S|}\det C^S$, where the $q$-th column of $C^S$ is
$Mu_q$ if $q\in S$ and $u_q$ if $q\notin S$. Expanding $\det C^S$ along the columns $u_q$, $q\notin S$, gives
$\det C^S=\det M[S|S]$.
\end{proof}

\begin{lemma}[compression]\label{lem.compress}
Let $Q\subseteq\{1,\dots,N\}$ and $M^\circ=M[Q|Q]$. Then
\[
   \sum_{S\subseteq Q}(-1)^{|S|}\,\langle u_S,\ext(M)u_S\rangle=\det(1-M^\circ).
\]
\end{lemma}
\begin{proof}
For $S\subseteq Q$ we have $\langle u_S,\ext(M)u_S\rangle=\det M[S|S]=\det M^\circ[S|S]$; apply
Lemma~\ref{lem.str} to $M^\circ$.
\end{proof}

\begin{lemma}[rank-one perturbations]\label{lem.rankone}
Let $M\in\End(U)$, let $p$ be a column vector and $\ell^{\T}$ a row vector, and let $z$ be an indeterminate. Then
every matrix entry of $\ext(M+z\,p\ell^{\T})$ is a polynomial of degree $\le1$ in $z$. In particular
$\det(M+z\,p\ell^{\T})$ has degree $\le 1$ in $z$.
\end{lemma}
\begin{proof}
By Lemma~\ref{lem.minors} (ii) the entries are the minors $\det\bigl(M[T|S]+z\,p_T\ell_S^{\T}\bigr)$, where $p_T$,
$\ell_S$ are the restrictions of $p$, $\ell$. The $m$-th column of this matrix is (the $m$-th column of $M[T|S]$)
$+\,z\,\ell_{s_m}p_T$. Expanding by multilinearity, every term containing two or more columns proportional to
$p_T$ vanishes.
\end{proof}

\subsection{Generalizing the Kauffman--Saleur construction: the Fock space of \texorpdfstring{$n$}{n} strands}\label{ssec.fock}
Let $U_n$ be the $\BK$-vector space with ordered basis
\[
  (u_1,u_2,\dots,u_{2n})=(a_1,b_1,a_2,b_2,\dots,a_n,b_n),
\]
and $\CF_n=\ext(U_n)$. We call $a_k,b_k$ the two \emph{fermions on the $k$-th strand}.
Put $\omega_1(a,b)=1$, $\omega_2(a,b)=a$, $\omega_3(a,b)=b$, $\omega_4(a,b)=a\wedge b$ and define the linear map
\begin{equation}\label{eq.Phi}
  \Phi_n\colon W^{\otimes n}\to\CF_n,\qquad
  \Phi_n(e_{m_1}\otimes\dots\otimes e_{m_n})=\omega_{m_1}(a_1,b_1)\wedge\omega_{m_2}(a_2,b_2)\wedge\dots\wedge\omega_{m_n}(a_n,b_n).
\end{equation}
Thus $e_1$ is the vacuum, $e_2$ and $e_3$ are the states with one fermion $a$, resp.\ $b$, and $e_4$ is the doubly
occupied state.

\begin{lemma}\label{lem.fock}
Let $n\ge1$.
\begin{enumerate}\renewcommand{\theenumi}{\roman{enumi}}
\item $\Phi_n$ maps the basis $(e_I)$ bijectively onto the basis $(u_S)$ of $\CF_n$; hence it is an isomorphism.
      For $I=(m_1,\dots,m_n)$ the element $\Phi_n(e_I)$ has degree $\sum_k\bigl(\a(e_{m_k})+\b(e_{m_k})\bigr)$, and
      $\Phi_n(e_I)=u_S$ with $S\subseteq Q:=\{3,4,\dots,2n\}$ if and only if $m_1=1$.
\item \textup{(Locality)} Let $M_0\in\End(U_2)$ be a $4\times4$ matrix, let $1\le k\le n-1$, and let
      $M_0^{(k)}\in\End(U_n)$ act as $M_0$ on $(a_k,b_k,a_{k+1},b_{k+1})$ and as the identity on the other
      basis vectors. Then
      \[
         \Phi_n^{-1}\,\ext\bigl(M_0^{(k)}\bigr)\,\Phi_n=\id_W^{\otimes(k-1)}\otimes\bigl(\Phi_2^{-1}\ext(M_0)\Phi_2\bigr)\otimes\id_W^{\otimes(n-k-1)} .
      \]
\end{enumerate}
\end{lemma}
\begin{proof}
(i) The wedge product in \eqref{eq.Phi} is a product of basis vectors $u_q$ with increasing indices, so it is
$u_S$ where $S$ records which of $a_k,b_k$ occur; conversely $S$ determines $I$. The rest is clear.

(ii) Write $\Phi_n(e_I)=\omega_<\wedge\omega_{\rm mid}\wedge\omega_>$, where $\omega_<$ is a product of the
$a_p,b_p$ with $p<k$, $\omega_>$ is a product of those with $p>k+1$, and
\[
   \omega_{\rm mid}=\omega_{m_k}(a_k,b_k)\wedge\omega_{m_{k+1}}(a_{k+1},b_{k+1}).
\]
Since $\ext(M_0^{(k)})$ is an algebra homomorphism which fixes the $a_p,b_p$ with $p\notin\{k,k+1\}$,
\[
  \ext\bigl(M_0^{(k)}\bigr)\bigl(\omega_<\wedge\omega_{\rm mid}\wedge\omega_>\bigr)
  =\omega_<\wedge\ext\bigl(M_0^{(k)}\bigr)(\omega_{\rm mid})\wedge\omega_> ,
\]
Moreover $\ext(M_0^{(k)})(\omega_{\rm mid})$ is obtained from $\ext(M_0)\Phi_2(e_{m_k}\otimes e_{m_{k+1}})$ by renaming
$a_1,b_1,a_2,b_2$ as $a_k,b_k,a_{k+1},b_{k+1}$. This is the claim.
\end{proof}

For $n\times n$ matrices $A$ and $B$ we write $A\oplus B\in\End(U_n)$ for the endomorphism which acts by $A$ on
$(a_1,\dots,a_n)$ and by $B$ on $(b_1,\dots,b_n)$:
\[
   (A\oplus B)\,a_q=\sum_pA_{pq}\,a_p,\qquad (A\oplus B)\,b_q=\sum_pB_{pq}\,b_p .
\]
Clearly $(A_1\oplus B_1)(A_2\oplus B_2)=A_1A_2\oplus B_1B_2$, and for $Q=\{3,4,\dots,2n\}$, $$(A\oplus B)[Q|Q]=A'\oplus B'$$ in the sense that it is the
corresponding endomorphism of $\spn(a_k,b_k\,;\,k\ge2)$ built from the matrices $A',B'$ with first row and column
removed. In particular
\begin{equation}\label{eq.detsum}
   \det\bigl(1-(A\oplus B)[Q|Q]\bigr)=\det(1-A')\,\det(1-B') .
\end{equation}

\section{Expressing the \texorpdfstring{$R$}{R}-matrix in the Fock space}\label{sec.rmatrix}

\subsection{The free two-flavor operators}
Let
\begin{equation}\label{eq.rho}
  \rho(x)=\begin{pmatrix}1-x&x\\1&0\end{pmatrix},\qquad
  \rho^-(x)=\begin{pmatrix}0&1\\x&1-x\end{pmatrix},\qquad \rho(x)\,\rho^-(x^{-1})=1 .
\end{equation}
$\rho(t)$ is the Burau matrix of a positive crossing and $\rho^-(t^{-1})=\rho(t)^{-1}$ that of a negative
crossing. With the notation of Section~\ref{ssec.fock} for $n=2$, define
\begin{equation}\label{eq.G}
  \Ga(x,y)=\Phi_2^{-1}\,\ext\bigl(\rho(x)\oplus\rho(y)^{\T}\bigr)\,\Phi_2,\qquad
  \Gm(x,y)=\Phi_2^{-1}\,\ext\bigl(\rho^-(x)\oplus\rho^-(y)^{\T}\bigr)\,\Phi_2 .
\end{equation}
These are endomorphisms of $W\otimes W$ whose matrix entries are polynomials in $x,y$. Explicitly,
$\Ga(x,y)$ is $\ext(M)$ and $\Gm(x,y)$ is $\ext(M^-)$ for the endomorphisms of $U_2=\spn(a_1,b_1,a_2,b_2)$
\begin{equation}\label{eq.MM}
\begin{aligned}
 Ma_1&=(1-x)a_1+a_2, & Ma_2&=x\,a_1, &\qquad M^-a_1&=x\,a_2, & M^-a_2&=a_1+(1-x)a_2,\\
 Mb_1&=(1-y)b_1+y\,b_2, & Mb_2&=b_1, &\qquad M^-b_1&=b_2, & M^-b_2&=y\,b_1+(1-y)b_2 .
\end{aligned}
\end{equation}
By \eqref{eq.rho} and Lemma~\ref{lem.minors} (i),
\begin{equation}\label{eq.Ginv}
   \Gm(x,y)=\Ga(x^{-1},y^{-1})^{-1}.
\end{equation}
Since $\Ga(x,y)$ is the exterior power of the direct sum of two Burau matrices, it is exactly the tensor
product (in the graded sense) of two copies of the Kauffman--Saleur $R$-matrix \cite[(5.9)]{KauffmanSaleur}, with
parameters $x$ and $y$, the second one transposed.

\subsection{Statement}
Let $K\in\End(W\otimes W)$ be the diagonal operator
\begin{equation}\label{eq.K}
   K\,e_{ij}=t_1^{\b(e_j)}\,e_{ij},\qquad\text{and put}\qquad \Rp=-K\,R\,K^{-1}\qquad(R=R_{-1}).
\end{equation}
For a polynomial $f=f(x,y)=f_{00}+x\,f_{10}+y\,f_{01}+xy\,f_{11}$ of degree $\le1$ in each of $x$ and $y$, with
coefficients $f_{ab}$ in a $\BK$-vector space, and for $u,v\in\BK$, define
\begin{equation}\label{eq.E1}
   \BE_{u,v}[f]=f_{00}+u\,f_{10}+v\,f_{01}+f_{11}\;=\;f(u,v)+(1-uv)\,f_{11}.
\end{equation}
Thus $\BE_{u,v}$ is the linear map uniquely determined by $1\mapsto1$, $x\mapsto u$, $y\mapsto v$ and $xy\mapsto1$.
It is applied to matrices entry by entry.

\begin{theorem}\label{thm.A}
All matrix entries of $\Ga(x,y)$ and $\Gm(x,y)$ have degree $\le1$ in each of $x$ and $y$, and
\begin{enumerate}\renewcommand{\theenumi}{\roman{enumi}}
\item $\Rp=\BE_{t_0,t_1}\bigl[\Ga(x,y)\bigr]$;
\item $\Rp^{-1}=\BE_{t_0^{-1},t_1^{-1}}\bigl[\Gm(x,y)\bigr]$.
\end{enumerate}
\end{theorem}

\subsection{Proof of Theorem~\ref{thm.A} \textup{(i)}}
We compute $\Ga(x,y)e_{ij}$ for the sixteen basis vectors, using \eqref{eq.Phi}, \eqref{eq.MM} and the rule
$\ext(M)(v_1\wedge\dots\wedge v_k)=Mv_1\wedge\dots\wedge Mv_k$. Recall that
\[
\begin{gathered}
 \Phi_2(e_{21})=a_1,\quad \Phi_2(e_{12})=a_2,\quad \Phi_2(e_{31})=b_1,\quad \Phi_2(e_{13})=b_2,\\
 \Phi_2(e_{22})=a_1\wedge a_2,\quad \Phi_2(e_{33})=b_1\wedge b_2,\quad \Phi_2(e_{41})=a_1\wedge b_1,\quad \Phi_2(e_{23})=a_1\wedge b_2,\\
 \Phi_2(e_{32})=b_1\wedge a_2,\quad \Phi_2(e_{14})=a_2\wedge b_2,\quad
 \Phi_2(e_{42})=a_1\wedge b_1\wedge a_2,\quad \Phi_2(e_{24})=a_1\wedge a_2\wedge b_2,\\
 \Phi_2(e_{43})=a_1\wedge b_1\wedge b_2,\quad \Phi_2(e_{34})=b_1\wedge a_2\wedge b_2,
\end{gathered}
\]
$\Phi_2(e_{11})=1$ and $\Phi_2(e_{44})=a_1\wedge b_1\wedge a_2\wedge b_2$.

\smallskip\noindent\emph{Degrees $0$ and $1$.} These are read off from \eqref{eq.MM}:
$\Ga e_{11}=e_{11}$, $\Ga e_{21}=(1-x)e_{21}+e_{12}$, $\Ga e_{12}=x\,e_{21}$, $\Ga e_{31}=(1-y)e_{31}+y\,e_{13}$,
$\Ga e_{13}=e_{31}$.

\smallskip\noindent\emph{Degree $2$.} $Ma_1\wedge Ma_2=((1-x)a_1+a_2)\wedge x\,a_1=-x\,a_1\wedge a_2$ and
$Mb_1\wedge Mb_2=((1-y)b_1+y\,b_2)\wedge b_1=-y\,b_1\wedge b_2$, so $\Ga e_{22}=-x\,e_{22}$,
$\Ga e_{33}=-y\,e_{33}$. Next, using $a_2\wedge b_1=-\Phi_2(e_{32})$,
\begin{align*}
 Ma_1\wedge Mb_1&=\bigl((1-x)a_1+a_2\bigr)\wedge\bigl((1-y)b_1+y\,b_2\bigr)\\
   &=(1-x)(1-y)\,a_1\wedge b_1+(1-x)y\,a_1\wedge b_2+(1-y)\,a_2\wedge b_1+y\,a_2\wedge b_2,\\
 Ma_1\wedge Mb_2&=\bigl((1-x)a_1+a_2\bigr)\wedge b_1=(1-x)\,a_1\wedge b_1+a_2\wedge b_1,\\
 Mb_1\wedge Ma_2&=\bigl((1-y)b_1+y\,b_2\bigr)\wedge x\,a_1=-x(1-y)\,a_1\wedge b_1-xy\,a_1\wedge b_2,\\
 Ma_2\wedge Mb_2&=x\,a_1\wedge b_1 .
\end{align*}

\smallskip\noindent\emph{Degrees $3$ and $4$.} Using $Ma_1\wedge Ma_2=-x\,a_1\wedge a_2$,
$Mb_1\wedge Mb_2=-y\,b_1\wedge b_2$, and the sign rules
$a_1\wedge a_2\wedge b_1=-\Phi_2(e_{42})$, $a_2\wedge b_1\wedge b_2=-\Phi_2(e_{34})$,
$b_2\wedge a_1\wedge b_1=\Phi_2(e_{43})$:
\begin{align*}
 Ma_1\wedge Mb_1\wedge Ma_2&=-(Ma_1\wedge Ma_2)\wedge Mb_1=x\,a_1\wedge a_2\wedge\bigl((1-y)b_1+y\,b_2\bigr)\\
     &=-x(1-y)\,\Phi_2(e_{42})+xy\,\Phi_2(e_{24}),\\
 Ma_1\wedge Ma_2\wedge Mb_2&=-x\,a_1\wedge a_2\wedge b_1=x\,\Phi_2(e_{42}),\\
 Ma_1\wedge Mb_1\wedge Mb_2&=-y\,\bigl((1-x)a_1+a_2\bigr)\wedge b_1\wedge b_2
     =-(1-x)y\,\Phi_2(e_{43})+y\,\Phi_2(e_{34}),\\
 Mb_1\wedge Ma_2\wedge Mb_2&=\bigl((1-y)b_1+y\,b_2\bigr)\wedge x\,a_1\wedge b_1=xy\,b_2\wedge a_1\wedge b_1=xy\,\Phi_2(e_{43}),
\end{align*}
and $Ma_1\wedge Mb_1\wedge Ma_2\wedge Mb_2=\det(M)\,a_1\wedge b_1\wedge a_2\wedge b_2=xy\,\Phi_2(e_{44})$, since
$\det M=\det\rho(x)\det\rho(y)=xy$.
The results are collected in the second column of Table~\ref{tab.G}; in particular every coefficient has
degree $\le1$ in $x$ and in $y$. The third column is obtained from the second by applying $\BE_{t_0,t_1}$, i.e.\
by the substitutions $x\mapsto t_0$, $y\mapsto t_1$, $xy\mapsto1$ in the expanded coefficients; for example
$(1-x)(1-y)=1-x-y+xy\mapsto2-t_0-t_1$, $(1-x)y\mapsto t_1-1$, $-x(1-y)\mapsto1-t_0$.

\begin{table}[ht]
\renewcommand{\arraystretch}{1.3}
\begin{tabular}{@{}lll@{}}
\hline
$v$ & $\Ga(x,y)\,v$ & $\BE_{t_0,t_1}[\Ga(x,y)]\,v$\\
\hline
$e_{11}$ & $e_{11}$ & $e_{11}$\\
$e_{21}$ & $(1-x)\,e_{21}+e_{12}$ & $(1-t_0)\,e_{21}+e_{12}$\\
$e_{12}$ & $x\,e_{21}$ & $t_0\,e_{21}$\\
$e_{31}$ & $(1-y)\,e_{31}+y\,e_{13}$ & $(1-t_1)\,e_{31}+t_1\,e_{13}$\\
$e_{13}$ & $e_{31}$ & $e_{31}$\\
$e_{22}$ & $-x\,e_{22}$ & $-t_0\,e_{22}$\\
$e_{33}$ & $-y\,e_{33}$ & $-t_1\,e_{33}$\\
$e_{41}$ & $(1-x)(1-y)\,e_{41}+(1-x)y\,e_{23}$ & $(2-t_0-t_1)\,e_{41}+(t_1-1)(e_{23}+e_{32})$\\[-3pt]
         & $\qquad{}-(1-y)\,e_{32}+y\,e_{14}$ & $\qquad{}+t_1\,e_{14}$\\
$e_{23}$ & $(1-x)\,e_{41}-e_{32}$ & $(1-t_0)\,e_{41}-e_{32}$\\
$e_{32}$ & $-x(1-y)\,e_{41}-xy\,e_{23}$ & $(1-t_0)\,e_{41}-e_{23}$\\
$e_{14}$ & $x\,e_{41}$ & $t_0\,e_{41}$\\
$e_{42}$ & $-x(1-y)\,e_{42}+xy\,e_{24}$ & $(1-t_0)\,e_{42}+e_{24}$\\
$e_{24}$ & $x\,e_{42}$ & $t_0\,e_{42}$\\
$e_{43}$ & $-(1-x)y\,e_{43}+y\,e_{34}$ & $(1-t_1)\,e_{43}+t_1\,e_{34}$\\
$e_{34}$ & $xy\,e_{43}$ & $e_{43}$\\
$e_{44}$ & $xy\,e_{44}$ & $e_{44}$\\
\hline
\end{tabular}
\smallskip
\caption{The free two-flavor operator $\Ga(x,y)=\Phi_2^{-1}\ext(\rho(x)\oplus\rho(y)^{\T})\Phi_2$ and its image under
$\BE_{t_0,t_1}$.}\label{tab.G}
\end{table}

On the other hand, $K^{-1}e_{ij}=t_1^{-\b(e_j)}e_{ij}$, so by \eqref{eq.K}
\begin{equation}\label{eq.Rprule}
  \text{the coefficient of $e_{kl}$ in $\Rp\,e_{ij}$ is $\;-t_1^{\,\b(e_l)-\b(e_j)}\;\times$ the coefficient of
  $e_{kl}$ in $R\,e_{ij}$,}
\end{equation}
where $\b(e_l)=1$ for $l\in\{3,4\}$ and $0$ for $l\in\{1,2\}$. Setting $r=-1$ in \eqref{eq.R}--\eqref{eq.R2}
and applying \eqref{eq.Rprule} gives Table~\ref{tab.Rp}. Its last column coincides with the last column of
Table~\ref{tab.G}. This proves~(i).

\begin{table}[ht]
\renewcommand{\arraystretch}{1.3}
\begin{tabular}{@{}lll@{}}
\hline
$v$ & $R\,v$ \ ($r=-1$) & $\Rp\,v=-K R K^{-1}v$\\
\hline
$e_{11}$ & $-e_{11}$ & $e_{11}$\\
$e_{21}$ & $-e_{12}+(t_0-1)\,e_{21}$ & $(1-t_0)\,e_{21}+e_{12}$\\
$e_{12}$ & $-t_0\,e_{21}$ & $t_0\,e_{21}$\\
$e_{31}$ & $-e_{13}+(t_1-1)\,e_{31}$ & $(1-t_1)\,e_{31}+t_1\,e_{13}$\\
$e_{13}$ & $-t_1\,e_{31}$ & $e_{31}$\\
$e_{22}$ & $t_0\,e_{22}$ & $-t_0\,e_{22}$\\
$e_{33}$ & $t_1\,e_{33}$ & $-t_1\,e_{33}$\\
$e_{41}$ & $-e_{14}+(t_1^{-1}-1)\,e_{23}-(t_1-1)\,e_{32}$ & $(2-t_0-t_1)\,e_{41}+(t_1-1)(e_{23}+e_{32})$\\[-3pt]
         & $\qquad{}+(t_0+t_1-2)\,e_{41}$ & $\qquad{}+t_1\,e_{14}$\\
$e_{23}$ & $t_1\,e_{32}+(t_0-1)t_1\,e_{41}$ & $(1-t_0)\,e_{41}-e_{32}$\\
$e_{32}$ & $t_1^{-1}\,e_{23}-(1-t_0)\,e_{41}$ & $(1-t_0)\,e_{41}-e_{23}$\\
$e_{14}$ & $-t_0t_1\,e_{41}$ & $t_0\,e_{41}$\\
$e_{42}$ & $-t_1^{-1}\,e_{24}+(t_0-1)\,e_{42}$ & $(1-t_0)\,e_{42}+e_{24}$\\
$e_{24}$ & $-t_0t_1\,e_{42}$ & $t_0\,e_{42}$\\
$e_{43}$ & $-t_1\,e_{34}+(t_1-1)\,e_{43}$ & $(1-t_1)\,e_{43}+t_1\,e_{34}$\\
$e_{34}$ & $-e_{43}$ & $e_{43}$\\
$e_{44}$ & $-e_{44}$ & $e_{44}$\\
\hline
\end{tabular}
\smallskip
\caption{The $R$-matrix \eqref{eq.R}--\eqref{eq.R2} at $r=-1$, and $\Rp=-KRK^{-1}$ computed by the rule
\eqref{eq.Rprule}.}\label{tab.Rp}
\end{table}

\subsection{Proof of Theorem~\ref{thm.A} \textup{(ii)}}
The same computation with $M^-$ in place of $M$ (see \eqref{eq.MM}) gives
\begin{equation}\label{eq.Gmlist}
\begin{aligned}
\Gm e_{11}&=e_{11}, & \Gm e_{44}&=xy\,e_{44},\\
\Gm e_{22}&=-x\,e_{22}, & \Gm e_{33}&=-y\,e_{33},\\
\Gm e_{21}&=x\,e_{12}, & \Gm e_{12}&=e_{21}+(1-x)\,e_{12},\\
\Gm e_{31}&=e_{13}, & \Gm e_{13}&=y\,e_{31}+(1-y)\,e_{13},\\
\Gm e_{41}&=x\,e_{14}, & \Gm e_{23}&=-xy\,e_{32}+x(1-y)\,e_{14},\\
\Gm e_{42}&=x\,e_{24}, & \Gm e_{24}&=xy\,e_{42}-x(1-y)\,e_{24},\\
\Gm e_{43}&=xy\,e_{34}, & \Gm e_{34}&=y\,e_{43}-(1-x)y\,e_{34},\\
\Gm e_{32}&=-e_{23}-(1-x)\,e_{14}, & \Gm e_{14}&=y\,e_{41}+(1-y)\,e_{23}-(1-x)y\,e_{32}\\
 & & &\qquad{}+(1-x)(1-y)\,e_{14}.
\end{aligned}
\end{equation}
For instance
$M^-a_1\wedge M^-b_2=x\,a_2\wedge(y\,b_1+(1-y)b_2)=-xy\,\Phi_2(e_{32})+x(1-y)\,\Phi_2(e_{14})$,
$M^-b_1\wedge M^-a_2=b_2\wedge(a_1+(1-x)a_2)=-\Phi_2(e_{23})-(1-x)\Phi_2(e_{14})$, and
\begin{align*}
 M^-b_1\wedge M^-a_2\wedge M^-b_2&=b_2\wedge\bigl(a_1+(1-x)a_2\bigr)\wedge y\,b_1
   =y\,a_1\wedge b_1\wedge b_2+(1-x)y\,a_2\wedge b_1\wedge b_2\\
   &=y\,\Phi_2(e_{43})-(1-x)y\,\Phi_2(e_{34}).
\end{align*}
Again all coefficients have degree $\le1$ in $x$ and in $y$. Write $\bar t_0=t_0^{-1}$, $\bar t_1=t_1^{-1}$ and
$\Rm=\BE_{\bar t_0,\bar t_1}[\Gm(x,y)]$; we have to show that $\Rm\Rp=1$. Both $\Rm$ and $\Rp$ preserve the nine sectors \eqref{eq.sectors}; in the ordered bases
\eqref{eq.sectors} their blocks are, by \eqref{eq.Gmlist} and Table~\ref{tab.Rp}
(columns are images of basis vectors):
\[
\renewcommand{\arraystretch}{1.25}
\begin{array}{@{}l@{\qquad}l@{\qquad}l@{}}
\text{sector} & \Rp & \Rm\\[8pt]
(e_{21},e_{12}),\ (e_{42},e_{24}) & \begin{pmatrix}1-t_0&t_0\\1&0\end{pmatrix} & \begin{pmatrix}0&1\\\bar t_0&1-\bar t_0\end{pmatrix}\\[16pt]
(e_{31},e_{13}),\ (e_{43},e_{34}) & \begin{pmatrix}1-t_1&1\\t_1&0\end{pmatrix} & \begin{pmatrix}0&\bar t_1\\1&1-\bar t_1\end{pmatrix}\\[16pt]
(e_{22}),\ (e_{33}),\ (e_{11}),\ (e_{44}) & -t_0,\ \ -t_1,\ \ 1,\ \ 1 & -\bar t_0,\ \ -\bar t_1,\ \ 1,\ \ 1
\end{array}
\]
and on the sector with basis $(e_{41},e_{23},e_{32},e_{14})$
\begin{equation}\label{eq.PQ}
\begin{gathered}
 \Rp\big|_{(1,1)}=\mathsf P=\begin{pmatrix}2-t_0-t_1&1-t_0&1-t_0&t_0\\ t_1-1&0&-1&0\\ t_1-1&-1&0&0\\ t_1&0&0&0\end{pmatrix},\\[4pt]
 \Rm\big|_{(1,1)}=\mathsf P^-=\begin{pmatrix}0&0&0&\bar t_1\\ 0&0&-1&1-\bar t_1\\ 0&-1&0&1-\bar t_1\\ \bar t_0&\bar t_0-1&\bar t_0-1&2-\bar t_0-\bar t_1\end{pmatrix}.
\end{gathered}
\end{equation}
(For example, $x(1-y)\mapsto \bar t_0-1$, $-(1-x)y\mapsto1-\bar t_1$ and $-x(1-y)\mapsto1-\bar t_0$ under $\BE_{\bar t_0,\bar t_1}$.)
The $2\times2$ and $1\times1$ blocks of $\Rm$ are visibly inverse to those of $\Rp$, since $\bar t_0t_0=\bar t_1t_1=1$:
\begin{gather*}
 \begin{pmatrix}0&1\\\bar t_0&1-\bar t_0\end{pmatrix}\begin{pmatrix}1-t_0&t_0\\1&0\end{pmatrix}
   =\begin{pmatrix}1&0\\ \bar t_0(1-t_0)+1-\bar t_0&\bar t_0t_0\end{pmatrix}=1,\\[4pt]
 \begin{pmatrix}0&\bar t_1\\1&1-\bar t_1\end{pmatrix}\begin{pmatrix}1-t_1&1\\t_1&0\end{pmatrix}
   =\begin{pmatrix}\bar t_1t_1&0\\ 1-t_1+(1-\bar t_1)t_1&1\end{pmatrix}=1 .
\end{gather*}
It remains to check $\mathsf P^-\mathsf P=1$. Denote the rows of $\mathsf P$ by $\mathsf p_1,\dots,\mathsf p_4$. Then
\begin{align*}
 \text{row }1:&\quad \bar t_1\mathsf p_4=(\bar t_1t_1,0,0,0)=(1,0,0,0),\\
 \text{row }2:&\quad -\mathsf p_3+(1-\bar t_1)\mathsf p_4=\bigl(-(t_1-1)+(1-\bar t_1)t_1,\,1,\,0,\,0\bigr)=(0,1,0,0),\\
 \text{row }3:&\quad -\mathsf p_2+(1-\bar t_1)\mathsf p_4=\bigl(-(t_1-1)+(1-\bar t_1)t_1,\,0,\,1,\,0\bigr)=(0,0,1,0),\\
 \text{row }4:&\quad \bar t_0\mathsf p_1+(\bar t_0-1)(\mathsf p_2+\mathsf p_3)+(2-\bar t_0-\bar t_1)\mathsf p_4=(0,0,0,1);
\end{align*}
in the last line the four entries are
$\bar t_0(2-t_0-t_1)+2(\bar t_0-1)(t_1-1)+(2-\bar t_0-\bar t_1)t_1=0$, then twice $\bar t_0(1-t_0)-(\bar t_0-1)=0$, and $\bar t_0t_0=1$.
Hence $\Rm\Rp=1$, so $\Rm=\Rp^{-1}$. This proves~(ii) (and shows differently that $R$ is invertible). \qed

\subsection{Reformulations}
Write $\Ga(x,y)=\Ga_{00}+x\,\Ga_{10}+y\,\Ga_{01}+xy\,\Ga_{11}$ and similarly for $\Gm$. Let $D=\Ga_{11}$ and
$D^-=\Gm_{11}$. Reading off the coefficients of $xy$ in Table~\ref{tab.G} and in \eqref{eq.Gmlist}:
\begin{equation}\label{eq.D}
\begin{aligned}
 D\,e_{41}&=D\,e_{32}=e_{41}-e_{23}, & D\,e_{42}&=e_{42}+e_{24}, & D\,e_{43}&=D\,e_{34}=e_{43}, & D\,e_{44}&=e_{44},\\
 D^-e_{14}&=-D^-e_{23}=e_{14}+e_{32}, & D^-e_{24}&=e_{24}+e_{42}, & D^-e_{34}&=D^-e_{43}=e_{34}, & D^-e_{44}&=e_{44},
\end{aligned}
\end{equation}
and $D$, $D^-$ vanish on the remaining ten basis vectors.

To describe $D$ and $D^-$ in fermionic terms, let $\cdag_u$ be the operator $u\wedge(\,\cdot\,)$ on $\CF_2$ and let
$c_u$ be the contraction with $u\in\{a_1,b_1,a_2,b_2\}$, i.e.\ the unique linear map with
$c_u(u\wedge\omega)=\omega$ and $c_u(\omega)=0$ for all monomials $\omega$ in the basis vectors other than $u$
(``creation'' and ``annihilation'' operators), and $n_u=\cdag_uc_u$.

\begin{corollary}[the quartic term]\label{cor.interaction}
\begin{align*}
  \Rp&=\Phi_2^{-1}\ext\bigl(\rho(t_0)\oplus\rho(t_1)^{\T}\bigr)\Phi_2+(1-t_0t_1)\,D,\\
  \Rp^{-1}&=\Phi_2^{-1}\ext\bigl(\rho(t_0)^{-1}\oplus(\rho(t_1)^{\T})^{-1}\bigr)\Phi_2+\bigl(1-(t_0t_1)^{-1}\bigr)\,D^- ,
\end{align*}
where, with $X=\cdag_{a_1}(c_{a_1}-c_{a_2})=n_{a_1}-\cdag_{a_1}c_{a_2}$ and $Y=(\cdag_{b_1}-\cdag_{b_2})c_{b_1}=n_{b_1}-\cdag_{b_2}c_{b_1}$,
\[
   D=\Phi_2^{-1}\,XY\,\Phi_2,\qquad D^-=\Phi_2^{-1}\,X^-Y^-\,\Phi_2,\qquad
   X^-=n_{a_2}-\cdag_{a_2}c_{a_1},\quad Y^-=n_{b_2}-\cdag_{b_1}c_{b_2}.
\]
In particular $\Rp$ is of the form $\Phi_2^{-1}\ext(M_0)\Phi_2$ at $t_0t_1=1$.
\end{corollary}
\begin{proof}
The two displayed formulas are Theorem~\ref{thm.A} combined with the second expression in \eqref{eq.E1} and
with \eqref{eq.Ginv}, which gives $\Gm(t_0^{-1},t_1^{-1})=\Ga(t_0,t_1)^{-1}$. For the operator form, note that $X$
involves only $a_1,a_2$ and $Y$ only $b_1,b_2$, and that on monomials
\begin{align*}
  X\colon&\quad 1\mapsto0,\quad a_1\mapsto a_1,\quad a_2\mapsto-a_1,\quad a_1\wedge a_2\mapsto a_1\wedge a_2;\\
  Y\colon&\quad 1\mapsto0,\quad b_1\mapsto b_1-b_2,\quad b_2\mapsto0,\quad b_1\wedge b_2\mapsto b_1\wedge b_2 .
\end{align*}
Moreover $XY(\mu\wedge\nu)=X(\mu)\wedge Y(\nu)$ for monomials $\mu$ in $a_1,a_2$ and $\nu$ in $b_1,b_2$ (moving
$c_{b_1}$ and then $\cdag_{b_1}$ or $\cdag_{b_2}$ across $\mu$ produces the sign $(-1)^{|\mu|}$ twice). Writing the six relevant
basis vectors as $\pm\mu\wedge\nu$,
\[
\begin{gathered}
\Phi_2(e_{41})=a_1\wedge b_1,\quad \Phi_2(e_{32})=-a_2\wedge b_1,\quad \Phi_2(e_{42})=-a_1\wedge a_2\wedge b_1,\\
\Phi_2(e_{43})=a_1\wedge b_1\wedge b_2,\quad \Phi_2(e_{34})=-a_2\wedge b_1\wedge b_2,\quad \Phi_2(e_{44})=-a_1\wedge a_2\wedge b_1\wedge b_2,
\end{gathered}
\]
we get $XY\Phi_2(e_{41})=a_1\wedge(b_1-b_2)=\Phi_2(e_{41}-e_{23})$,
$XY\Phi_2(e_{32})=a_1\wedge(b_1-b_2)$,
$XY\Phi_2(e_{42})=-a_1\wedge a_2\wedge(b_1-b_2)=\Phi_2(e_{42}+e_{24})$,
$XY\Phi_2(e_{43})=\Phi_2(e_{43})$, $XY\Phi_2(e_{34})=a_1\wedge b_1\wedge b_2=\Phi_2(e_{43})$ and
$XY\Phi_2(e_{44})=\Phi_2(e_{44})$; on the other ten basis vectors $X$ or $Y$ vanishes. This is \eqref{eq.D}.
The computation for $D^-$ is the same with the two strands exchanged:
$X^-\colon a_2\mapsto a_2$, $a_1\mapsto-a_2$, $a_1\wedge a_2\mapsto a_1\wedge a_2$ and
$Y^-\colon b_2\mapsto b_2-b_1$, $b_1\mapsto0$, $b_1\wedge b_2\mapsto b_1\wedge b_2$.
\end{proof}

\begin{remark}[the free-fermion condition]\label{rem.ffc}
For one flavor, the Kauffman--Saleur free-fermion condition \cite[(5.3)]{KauffmanSaleur} says precisely that the vertex
operator, normalised so that the empty vertex has weight $1$, is $\ext$ of its restriction to one-particle states
\cite[(5.2)--(5.5)]{KauffmanSaleur}. By Table~\ref{tab.Rp} the restriction of $\Rp$ to
one-particle states is $\rho(t_0)\oplus\rho(t_1)^{\T}$, so Corollary~\ref{cor.interaction} measures the failure of the
free-fermion condition: it is the operator $(1-t_0t_1)D$ of rank $4$. For instance the doubly occupied
configuration $e_{44}\mapsto e_{44}$ has weight $1$ instead of $\det\rho(t_0)\det\rho(t_1)=t_0t_1$.
\end{remark}

\begin{corollary}[in the basis of {\cite[A.1]{GHKST}}]\label{cor.paperbasis}
For $r=-1$,
\[
   -R=\Phi_2^{-1}\ext\bigl(\rho(t_0)\oplus\rho(t_1)\bigr)\Phi_2+(1-t_0t_1)\,C,\qquad C=K^{-1}DK,
\]
where $C\,e_{41}=C\,e_{32}=e_{41}-t_1^{-1}e_{23}$, $C\,e_{42}=e_{42}+t_1^{-1}e_{24}$, $C\,e_{43}=C\,e_{34}=e_{43}$,
$C\,e_{44}=e_{44}$, and $C$ vanishes on the other basis vectors. In fermionic terms
$C=\Phi_2^{-1}\,(n_{a_1}-\cdag_{a_1}c_{a_2})(n_{b_1}-t_1^{-1}\cdag_{b_2}c_{b_1})\,\Phi_2$.
\end{corollary}
\begin{proof}
$K=\Phi_2^{-1}\ext(1\oplus\diag(1,t_1))\Phi_2$, because $K$ multiplies $e_{ij}$ by $t_1$ exactly when $b_2$ occurs in
$\Phi_2(e_{ij})$. Since $\diag(1,t_1)^{-1}\rho(t_1)^{\T}\diag(1,t_1)=\rho(t_1)$, Lemma~\ref{lem.minors}(i) gives
$K^{-1}\Ga(t_0,t_1)K=\Phi_2^{-1}\ext(\rho(t_0)\oplus\rho(t_1))\Phi_2$. Conjugating the first formula of
Corollary~\ref{cor.interaction} by $K^{-1}$ gives the claim; the formulas for $C$ follow from \eqref{eq.D} and
$K^{-1}DK\,e_{ij}=t_1^{\b(e_j)}K^{-1}(D\,e_{ij})$, and the fermionic form from $K^{-1}\cdag_{b_2}K=t_1^{-1}\cdag_{b_2}$
(conjugating by $K$ in $\CF_2$).
\end{proof}

\begin{remark}\label{rem.Lambda1}
The operator $\Phi_2^{-1}\ext(\rho(t_0)\oplus\rho(t_1))\Phi_2$ is the free two-flavor $R$-matrix; by
Lemma~\ref{lem.str} its braid closures give products of two Burau determinants, i.e.\
$\Delta(t_0)\Delta(t_1)$. One checks directly that it coincides with the matrix $R_{\Lambda_1}$ of
\cite[Appendix~A.2]{GHKST} at $r=-1$.
So Corollary~\ref{cor.paperbasis} reads: \emph{at $r=-1$ the matrix of \cite[A.1]{GHKST} is minus the matrix of
\cite[A.2]{GHKST} minus $(1-t_0t_1)C$}, where $C$ has only nine non-zero entries.
\end{remark}

\section{A determinant formula}\label{sec.theorem}

Throughout this section and the next one, $w=s_1\cdots s_c$, $s_j=\sigma_{\nu_j}^{\ve_j}$, is a braid word on $n$
strands, $L=\widehat w$ is its closure, and $\mu$ is the number of components of $L$. Recall that $V(w)$ denotes the
state sum of Definition~\ref{def.V} at $r=-1$, and that it is the $V_1$-polynomial of $L$
(Theorem~\ref{thm.GHKST} and Remark~\ref{rem.convention}). Moreover, the permutation of $\{1,\dots,n\}$ induced by
$w$ is a product of $c$ transpositions and has $\mu$ cycles, so its sign is $(-1)^c=(-1)^{n-\mu}$. Thus
\begin{equation}\label{eq.VV1}
   V_1(L;t_0,t_1)=V(w)\qquad\text{and}\qquad(-1)^{n-1+c}=(-1)^{\mu-1}.
\end{equation}
All the results of Sections~\ref{sec.theorem} and~\ref{sec.lagrangian} are stated for $V_1(L)=V_1(L;t_0,t_1)$. Their
proofs compute the state sum $V(w)$ directly from Definition~\ref{def.V}, and the passage to $V_1(L)$ is always
\eqref{eq.VV1}.

\subsection{Burau matrices with a variable at each crossing}
Attach to the $j$-th crossing of $w$ two indeterminates $x_j$ and $y_j$. With $\rho^{+}=\rho$ and $\rho^-$ as in \eqref{eq.rho}, define the
$n\times n$ matrices
\begin{equation}\label{eq.Bj}
  B_j(x_j)=1_{\nu_j-1}\oplus\rho^{\ve_j}(x_j)\oplus1_{n-\nu_j-1},\qquad \tB_j(y_j)=B_j(y_j)^{\T}
\end{equation}
(block diagonal matrices, the $2\times2$ block in rows and columns $\nu_j,\nu_j+1$), and
\begin{equation}\label{eq.B}
  B(x)=B_c(x_c)\cdots B_2(x_2)B_1(x_1),\qquad \tB(y)=\tB_c(y_c)\cdots\tB_2(y_2)\tB_1(y_1),
\end{equation}
\begin{equation}\label{eq.P}
  P_w(x)=\det\bigl(1-B(x)'\bigr),\qquad \tP_w(y)=\det\bigl(1-\tB(y)'\bigr),
\end{equation}
where $'$ deletes the first row and column. If all $x_j$ are specialised to $t^{\ve_j}$, then
$B_j$ is the usual Burau matrix of the letter $s_j$, $B$ is the (unreduced) Burau matrix of
$w$, and $P_w$ is the classical minor $\det(1-B')$ computing the Alexander polynomial; see
Remark~\ref{rem.burau}. Note that the product $\tB(y)$ is taken in the same order as $B(x)$, so that
$\tB(y)$ is \emph{not} the transpose of $B(y)$.

\begin{lemma}\label{lem.multiaffine}
$P_w(x)$ has degree $\le1$ in each variable $x_j$, and $\tP_w(y)$ has degree $\le1$ in each $y_j$.
\end{lemma}
\begin{proof}
$\rho(z)=\begin{psmallmatrix}1&0\\1&0\end{psmallmatrix}+z\begin{psmallmatrix}-1&1\\0&0\end{psmallmatrix}$ and
$\rho^-(z)=\begin{psmallmatrix}0&1\\0&1\end{psmallmatrix}+z\begin{psmallmatrix}0&0\\1&-1\end{psmallmatrix}$, and the
matrices multiplying $z$ have rank one. Hence, by Lemma~\ref{lem.rankone}, the entries of $\ext(B_j(x_j))$ have
degree $\le1$ in $x_j$, and the entries of $\ext(B(x))=\ext(B_c(x_c))\cdots\ext(B_1(x_1))$ have degree $\le1$ in each
$x_j$. By Lemma~\ref{lem.str}, $P_w(x)=\sum_{S\subseteq\{2,\dots,n\}}(-1)^{|S|}\det B(x)[S|S]$ is a signed sum of
such entries. The same argument applies to $\tP_w$.
\end{proof}

Let $\CA$ be the $\BK$-vector space of polynomials in $x_1,\dots,x_c,y_1,\dots,y_c$ of degree $\le1$ in each
variable; it has the basis $x^Sy^{S'}=\prod_{j\in S}x_j\prod_{j\in S'}y_j$, $S,S'\subseteq\{1,\dots,c\}$.
Define the $\BK$-linear map
\begin{equation}\label{eq.E}
  \BE=\BE_w\colon\CA\to\BK,\qquad
  \BE\bigl[x^Sy^{S'}\bigr]=\prod_{j\in S\smallsetminus S'}t_0^{\ve_j}\ \prod_{j\in S'\smallsetminus S}t_1^{\ve_j}.
\end{equation}
In words: $x_j\mapsto t_0^{\ve_j}$ and $y_j\mapsto t_1^{\ve_j}$, except that the product $x_jy_j$ of the two
variables of the \emph{same} crossing is replaced by $1$. By construction, if $f_j$ is a polynomial of degree $\le1$
in each of $x_j,y_j$ (and involves no other variable), then
\begin{equation}\label{eq.Emult}
   \BE\Bigl[\,\prod_{j=1}^cf_j\Bigr]=\prod_{j=1}^c\BE_{t_0^{\ve_j},\,t_1^{\ve_j}}[f_j],
\end{equation}
with the one-crossing functional \eqref{eq.E1}: both sides are multilinear in $(f_1,\dots,f_c)$ and agree on
monomials.

\subsection{Main theorem}

\begin{theorem}\label{thm.B}
Let $w$ be a braid word on $n$ strands with $c$ crossings, and let $L$ be its closure, a link with $\mu$ components.
Then
\begin{equation}\label{eq.main}
   V_1(L;t_0,t_1)=(-1)^{\mu-1}\;\BE\bigl[P_w(x)\,\tP_w(y)\bigr]
     =(-1)^{\mu-1}\;\BE\Bigl[\det\bigl(1-B(x)'\bigr)\,\det\bigl(1-\tB(y)'\bigr)\Bigr].
\end{equation}
Equivalently, by \eqref{eq.VV1}, the state sum of Definition~\ref{def.V} is
$V(w)=(-1)^{n-1+c}\,\BE\bigl[P_w(x)\,\tP_w(y)\bigr]$.
\end{theorem}

\begin{proof}
By \eqref{eq.VV1} it suffices to prove the formula for the state sum $V(w)$. We do this in three steps.

\smallskip\emph{Step 1: one crossing.} Let $K_n\in\End(W^{\otimes n})$ be the diagonal operator
\[
   K_ne_I=t_1^{\sum_k(k-1)\b(e_{m_k})}\,e_I\qquad\text{for }I=(m_1,\dots,m_n),
\]
and let $\BB_j=B_j(x_j)\oplus\tB_j(y_j)\in\End(U_n)$ in the notation of Section~\ref{ssec.fock}. We claim that
\begin{equation}\label{eq.step1}
   -K_n\,\varrho(s_j)\,K_n^{-1}=\Phi_n^{-1}\ \BE_{t_0^{\ve_j},\,t_1^{\ve_j}}\bigl[\ext(\BB_j)\bigr]\ \Phi_n .
\end{equation}
Let $i=\nu_j$. The exponent of $t_1$ in $K_ne_I$ is
$(i-1)\bigl(\b(e_{m_i})+\b(e_{m_{i+1}})\bigr)+\b(e_{m_{i+1}})$ plus terms not involving the $i$-th and
$(i+1)$-st tensor factors. Since $R^{\pm1}$ preserves the total $\b$-degree of these two factors
(Lemma~\ref{lem.blocks}), we get $K_n\varrho(\sigma_i^{\pm1})K_n^{-1}=\id\otimes KR^{\pm1}K^{-1}\otimes\id$ with $K$ as in
\eqref{eq.K}, that is, $-K_n\varrho(\sigma_i^{\pm1})K_n^{-1}=\id\otimes\Rp^{\pm1}\otimes\id$. By
Theorem~\ref{thm.A}, $\Rp=\BE_{t_0,t_1}[\Ga(x_j,y_j)]$ and $\Rp^{-1}=\BE_{t_0^{-1},t_1^{-1}}[\Gm(x_j,y_j)]$. Finally,
by the locality Lemma~\ref{lem.fock} (ii) applied to $M_0=\rho^{\pm}(x_j)\oplus\rho^{\pm}(y_j)^{\T}$,
\[
  \id^{\otimes(i-1)}\otimes\Ga(x_j,y_j)\otimes\id^{\otimes(n-i-1)}=\Phi_n^{-1}\ext\bigl(M_0^{(i)}\bigr)\Phi_n ,\qquad M_0^{(i)}=B_j(x_j)\oplus\tB_j(y_j)=\BB_j
\]
for $\ve_j=+1$, and the same with $\Gm$ for $\ve_j=-1$. This proves \eqref{eq.step1}.

\smallskip\emph{Step 2: the whole braid.} Multiplying the identities \eqref{eq.step1} for $j=c,\dots,1$,
\begin{align*}
  (-1)^cK_n\,\varrho(w)\,K_n^{-1}
  &=\Phi_n^{-1}\ \BE_{t_0^{\ve_c},t_1^{\ve_c}}\bigl[\ext(\BB_c)\bigr]\cdots\BE_{t_0^{\ve_1},t_1^{\ve_1}}\bigl[\ext(\BB_1)\bigr]\ \Phi_n\\
  &=\Phi_n^{-1}\ \BE\bigl[\ext(\BB_c)\cdots\ext(\BB_1)\bigr]\ \Phi_n .
\end{align*}
The second equality is \eqref{eq.Emult}, applied to each matrix entry of the product: such an entry is a sum of
products of one entry from each factor, and the entries of $\ext(\BB_j)$ have degree $\le1$ in $x_j$ and in $y_j$
and involve no other variable. By Lemma~\ref{lem.minors} (i),
$\ext(\BB_c)\cdots\ext(\BB_1)=\ext(\BB)$ with $\BB=\BB_c\cdots\BB_1=B(x)\oplus\tB(y)$. So
\begin{equation}\label{eq.step2}
   (-1)^cK_n\,\varrho(w)\,K_n^{-1}=\Phi_n^{-1}\ \BE\bigl[\ext\bigl(B(x)\oplus\tB(y)\bigr)\bigr]\ \Phi_n .
\end{equation}

\smallskip\emph{Step 3: the closure.} Conjugation by the diagonal operator $K_n$ does not change diagonal matrix
entries. By Lemma~\ref{lem.fock} (i), the vectors $e_{(1,m_2,\dots,m_n)}$ correspond under $\Phi_n$ to the $u_S$ with
$S\subseteq Q=\{3,\dots,2n\}$, and by \eqref{eq.hparity} $h_{m_2}\cdots h_{m_n}=(-1)^{n-1}(-1)^{|S|}$. Therefore, by
Definition~\ref{def.V} and \eqref{eq.step2},
\[
   V(w)=(-1)^{n-1+c}\sum_{S\subseteq Q}(-1)^{|S|}\,\BE\bigl[\langle u_S,\ext(\BB)u_S\rangle\bigr]
       =(-1)^{n-1+c}\,\BE\bigl[\det\bigl(1-\BB[Q|Q]\bigr)\bigr]
\]
by Lemma~\ref{lem.compress} and the linearity of $\BE$, and
$\det(1-\BB[Q|Q])=\det(1-B(x)')\det(1-\tB(y)')=P_w(x)\tP_w(y)$ by \eqref{eq.detsum}. This is the asserted formula
for $V(w)$.
\end{proof}

\begin{remark}[reversing the word]\label{rem.reverse}
Let $w^{\rm rev}=s_c\cdots s_2s_1$ be the reversed word, with its crossings numbered so that the crossing $j$ is still
the letter $s_j$. By \eqref{eq.B},
\begin{align*}
   B_{w^{\rm rev}}(x)&=B_1(x_1)B_2(x_2)\cdots B_c(x_c)=\tB_w(x)^{\T},\\
   \tB_{w^{\rm rev}}(y)&=\tB_1(y_1)\tB_2(y_2)\cdots\tB_c(y_c)=B_w(y)^{\T},
\end{align*}
and deleting the first row and column commutes with transposition, so $P_{w^{\rm rev}}=\tP_w$ and
$\tP_{w^{\rm rev}}=P_w$. The closure of $w^{\rm rev}$ is the closure $L$ of $w$ with the orientation of all components
reversed; we denote it by $-L$.
Theorem~\ref{thm.B} for $w^{\rm rev}$, with the names of the variables $x$ and $y$ exchanged, therefore gives
\[
   V_1(-L;t_0,t_1)=V_1(L;t_1,t_0).
\]
Thus the symmetry
$V_1(L;t_0,t_1)=V_1(L;t_1,t_0)$ is equivalent to the invariance of $V_1$ under reversal of the orientation. This is consistent since the Links--Gould invariant does not distinguish a link from its inverse \cite{DeWitKauffmanLinks},
and $V_1$ is the Links--Gould invariant \cite{GHKKST}. In particular the direction in which braid words are read
(Section~\ref{sec.setup}) is immaterial for the value of $V_1$.
\end{remark}

\subsection{Examples}

\begin{example}[unknots]\label{ex.unknot}
For $w=\sigma_1$ ($n=2$): $B=\rho(x_1)$, $B'=(0)$, so $P_w=\tP_w=1$ and $V_1=1$.
For $w=\sigma_1^{-1}$: $B=\rho^-(x_1)$, $B'=(1-x_1)$, so $P_w=x_1$, $\tP_w=y_1$, and
$V_1=\BE[x_1y_1]=1$. Here the two Burau determinants alone would give $t_0^{-1}t_1^{-1}$; the rule
$x_1y_1\mapsto1$ is what makes the formula invariant under the Markov III move.
\end{example}

\begin{example}[Hopf link]
$w=\sigma_1^2$, $n=2$, $\mu=2$. $B=\rho(x_2)\rho(x_1)$ has lower right entry $x_1$, so $P_w=1-x_1$; similarly
$\tP_w=1-y_2$. No crossing contributes both variables, and $V_1=-(1-t_0)(1-t_1)$.
\end{example}

\begin{example}[trefoil]\label{ex.trefoil}
$w=\sigma_1^3$, $n=2$. The lower right entries of $\rho(x_3)\rho(x_2)\rho(x_1)$ and of
$\rho(y_3)^{\T}\rho(y_2)^{\T}\rho(y_1)^{\T}$ are $x_1(1-x_2)$ and $y_3(1-y_2)$, so
\[
   P_w=1-x_1+x_1x_2,\qquad \tP_w=1-y_3+y_2y_3 .
\]
The only monomial of $P_w\tP_w$ containing a product $x_jy_j$ is $x_1x_2\cdot y_2y_3\mapsto t_0t_1$ (instead of
$t_0^2t_1^2$). Hence
\[
   V_1(3_1)=(1-t_0+t_0^2)(1-t_1+t_1^2)+t_0t_1(1-t_0t_1)
           =1-t_0-t_1+t_0^2+t_1^2+2t_0t_1-t_0^2t_1-t_0t_1^2 .
\]
\end{example}

\begin{example}[figure-eight knot]
$w=\sigma_1\sigma_2^{-1}\sigma_1\sigma_2^{-1}$, $n=3$. One finds
\[
   P_w=x_2+x_4-x_1x_4-x_2x_4+x_1x_2x_4,\qquad \tP_w=y_2+y_4-y_2y_4-y_3y_4+y_2y_3y_4 ,
\]
and, with $\BE[x_1]=t_0$, $\BE[x_2]=\BE[x_4]=t_0^{-1}$, $\BE[y_3]=t_1$, $\BE[y_2]=\BE[y_4]=t_1^{-1}$, $\BE[x_jy_j]=1$,
\[
  V_1(4_1)=7-3\bigl(t_0+t_1+t_0^{-1}+t_1^{-1}\bigr)+2\bigl(t_0t_1+t_0^{-1}t_1^{-1}\bigr)+t_0t_1^{-1}+t_0^{-1}t_1 .
\]
\end{example}

\subsection{A consequence}
Put $\bx_j=t_0^{\ve_j}$, $\by_j=t_1^{\ve_j}$ and
\begin{equation}\label{eq.gj}
   g_j=1-(t_0t_1)^{\ve_j}=1-\bx_j\by_j .
\end{equation}

\begin{corollary}[expansion around the free theory]\label{cor.expansion}
Write $P_w(\bx+\d)=\sum_{S}a_S\,\d^S$ and $\tP_w(\by+\d')=\sum_Sb_S\,\d'^S$ with
$a_S\in\BZ[t_0^{\pm1}]$, $b_S\in\BZ[t_1^{\pm1}]$, the sums running over $S\subseteq\{1,\dots,c\}$. Then
\[
   V_1(L;t_0,t_1)=(-1)^{\mu-1}\sum_{S\subseteq\{1,\dots,c\}}\ \Bigl(\prod_{j\in S}g_j\Bigr)\,a_S\,b_S .
\]
The term $S=\emptyset$ is $P_w(\bx)\tP_w(\by)$, the product of the two ordinary Burau determinants at $t_0$ and
$t_1$; all other terms are divisible by $1-t_0t_1$.
\end{corollary}
\begin{proof}
With $x_j=\bx_j+\d_j$, $y_j=\by_j+\d'_j$ we have $\BE[\d_j]=\BE[\d'_j]=0$ and
$\BE[\d_j\d'_j]=\BE[x_jy_j]-\bx_j\by_j=g_j$. By \eqref{eq.Emult},
$\BE[\d^S\d'^{S'}]=\prod_{j\in S}g_j$ if $S=S'$ and $0$ otherwise. Hence
$\BE[P_w(x)\tP_w(y)]=\sum_S\bigl(\prod_{j\in S}g_j\bigr)a_Sb_S$, and the claim follows from Theorem~\ref{thm.B}.
\end{proof}

\begin{remark}[the free point and the Alexander polynomial]\label{rem.burau}
If $t_0t_1=1$ then all $g_j$ vanish and
$V_1(L;t,t^{-1})=(-1)^{\mu-1}P_w(t^{\ve})\,\tP_w(t^{-\ve})$, a product of two Burau determinants. By the classical
theorem of Burau, $\det(1-B')$ at $x_j=t^{\ve_j}$ equals the Alexander polynomial up to a
unit. More precisely, let $\Delta_L(t)=\nabla_L(t^{1/2}-t^{-1/2})$ be the Conway-normalised Alexander polynomial; it
satisfies $\Delta_L(t^{-1})=(-1)^{\mu-1}\Delta_L(t)$. Then
$P_w(t^{\ve})=(-1)^{\mu-1}\,t^{(e-n+1)/2}\Delta_L(t)$ with $e=\sum_j\ve_j$, the letters $\sigma_i$ being positive
crossings. Moreover $\tP_w(t^{\ve})=P_w(t^{\ve})$: indeed $\rho(t)^{\T}=\diag(1,t)\rho(t)\diag(1,t)^{-1}$, and the same holds
for $\rho^-(t^{-1})=\rho(t)^{-1}$; hence
$\tB(t^\ve)=T\,B(t^\ve)\,T^{-1}$ with $T=\diag(1,t,\dots,t^{n-1})$, and conjugation by a diagonal matrix commutes
with deleting the first row and column. So one recovers
$V_1(L;t,t^{-1})=(-1)^{\mu-1}\Delta_L(t)\Delta_L(t^{-1})=\Delta_L(t)^2$ \cite{Kohli}, and
Corollary~\ref{cor.expansion} says
\[
     V_1(L;t_0,t_1)=(-1)^{\mu-1}\,P_w(t_0^{\ve})\,P_w(t_1^{\ve})+(1-t_0t_1)(\cdots),\qquad
   P_w(t^{\ve})=(-1)^{\mu-1}\,t^{(e-n+1)/2}\Delta_L(t).
\]
\end{remark}

\section{Lagrangian forms: edges, Berezin integral and spanning trees}\label{sec.lagrangian}

In this section we rewrite Theorem~\ref{thm.B} in the style of \cite[Section~6]{KauffmanSaleur}: first as a determinant
indexed by the edges of the closed braid diagram, then as a Berezin integral and finally as a state sum over pairs of spanning trees.
Throughout, $w=s_1\cdots s_c$ is a braid word on $n$ strands such that every position $p\in\{1,\dots,n\}$ is
moved by at least one letter (i.e.\ $p\in\{\nu_j,\nu_j+1\}$ for some $j$). This is no restriction: otherwise the
closure contains a circle disjoint from the rest of the diagram and one sees directly from \eqref{eq.P} that
$P_w=0$ if $n\ge2$.\footnote{If position $p\ge2$ is not moved, the $p$-th row of $B(x)$ is the $p$-th unit vector,
so $1-B(x)'$ has a zero row. If position $1$ is not moved, $B(x)=1\oplus B_{\ge2}(x)$ and $1-B(x)'=1-B_{\ge2}(x)$
kills the vector $(1,\dots,1)^{\T}$, because every $\rho^{\pm}(x_j)$ fixes $(1,1)^{\T}$.}

\subsection{Two determinant lemmas}

\begin{lemma}[cyclic products]\label{lem.cyclic}
Let $C_1,\dots,C_c$ be $n\times n$ matrices. Let $\CC$ be the $nc\times nc$ matrix, with rows and columns
indexed by pairs $(p,j)$, $1\le p\le n$, $j\in\BZ/c$, defined by
$\CC[(p',j+1),(p,j)]=(C_j)_{p'p}$ and $\CC[(p',j'),(p,j)]=0$ if $j'\ne j+1$. Then
\[
   \det(1-\CC)=\det(1-C_cC_{c-1}\cdots C_1).
\]
\end{lemma}
\begin{proof}
We use the identity $\det\begin{psmallmatrix}A&B\\D&1\end{psmallmatrix}=\det(A-BD)$ for block matrices, which
follows from
$\begin{psmallmatrix}A&B\\D&1\end{psmallmatrix}\begin{psmallmatrix}1&0\\-D&1\end{psmallmatrix}
=\begin{psmallmatrix}A-BD&B\\0&1\end{psmallmatrix}$. For $c=1$ there is nothing to prove. Let $c\ge2$ and order the
index set so that the level $j=2$ comes last. In the corresponding block form of $1-\CC$, the last diagonal
block is $1$, the block $D$ (rows of level $2$) has the single non-zero entry $-C_1$ in the columns of level $1$,
and the block $B$ (columns of level $2$) has the single non-zero entry $-C_2$ in the rows of level $3$ (level $1$
if $c=2$). Hence $A-BD$ is $1-\CC^\flat$, where $\CC^\flat$ is the matrix of the same kind built from
the $c-1$ matrices $C_2C_1,C_3,\dots,C_c$. The claim follows by induction on $c$.
\end{proof}

\begin{lemma}[merging two nodes]\label{lem.merge}
Let $Z$ be a square matrix with rows and columns indexed by a finite set $\CN$, and let $\xi\ne\zeta$ in $\CN$ be
such that the column $\xi$ of $Z$ and the row $\zeta$ of $Z$ both have the single non-zero entry $Z[\zeta,\xi]=1$.
Let $Z^\flat$ be the matrix indexed by $\CN\smallsetminus\{\zeta\}$ with
\[
   Z^\flat[k,l]=Z[k,l]\ \ (l\ne\xi),\qquad Z^\flat[k,\xi]=Z[k,\zeta].
\]
Then $\det(1-Z)=\det(1-Z^\flat)$. Moreover, if $d\in\CN$ and the subscript $\hat d$ denotes deletion of the row and
column $d$, then $\det\bigl((1-Z)_{\hat d}\bigr)=\det\bigl((1-Z^\flat)_{\hat d}\bigr)$ for $d\notin\{\xi,\zeta\}$, and
$\det\bigl((1-Z)_{\hat d}\bigr)=\det\bigl((1-Z^\flat)_{\hat\xi}\bigr)$ for $d\in\{\xi,\zeta\}$.
\end{lemma}
\begin{proof}
The row $\zeta$ of $1-Z$ has the entries $1$ in column $\zeta$ and $-1$ in column $\xi$, and zeros elsewhere. Adding
column $\zeta$ to column $\xi$ does not change the determinant and makes row $\zeta$ equal to the $\zeta$-th unit row;
expanding along it removes row and column $\zeta$. For $k\ne\zeta$ the new entry in position $(k,\xi)$ is
$(\d_{k\xi}-Z[k,\xi])+(\d_{k\zeta}-Z[k,\zeta])=\d_{k\xi}-Z[k,\zeta]$, because $Z[k,\xi]=0$ for $k\ne\zeta$. This is
$(1-Z^\flat)[k,\xi]$, which proves the first claim. For $d\notin\{\xi,\zeta\}$ the same argument applies to the
matrices with row and column $d$ deleted. If $d=\xi$: in $(1-Z)_{\hat\xi}$ the row $\zeta$ is the $\zeta$-th unit row,
and expanding along it gives the determinant of $1-Z$ with rows and columns $\xi,\zeta$ deleted, which is
$(1-Z^\flat)_{\hat\xi}$. If $d=\zeta$: in $(1-Z)_{\hat\zeta}$ the column $\xi$ is the $\xi$-th unit column, with the
same conclusion.
\end{proof}

\subsection{The edge matrix of a closed braid}\label{ssec.edges}
Consider the set of \emph{nodes} $\CN=\{1,\dots,n\}\times\BZ/c$; the node $(p,j)$ stands for ``position $p$, just
before the letter $s_j$''. Let $\BT(x)$ be the matrix $\CC$ of Lemma~\ref{lem.cyclic} for
$C_j=B_j(x_j)$, i.e.
\[
   \BT(x)\bigl[(p',j+1),(p,j)\bigr]=B_j(x_j)_{p'p},
\]
and let $\widetilde{\BT}(y)$ be the same with $\tB_j(y_j)$.
A node $(p,j)$ is \emph{passive} if $p\notin\{\nu_j,\nu_j+1\}$ and \emph{active} otherwise. If $(p,j)$ is passive,
then the column $(p,j)$ and the row $(p,j+1)$ of $\BT(x)$ (and of $\widetilde\BT(y)$) have the single non-zero
entry $1$ in position $\bigl((p,j+1),(p,j)\bigr)$.

The \emph{edges} of the closed braid diagram are the equivalence classes of nodes for the equivalence relation
generated by $(p,j)\sim(p,j+1)$ for passive $(p,j)$. By our assumption on $w$, every edge is a chain
$(p,j),(p,j+1),\dots,(p,j+m)$ whose last node is active, all others being passive, and such that $(p,j-1)$ is active.
We say that this edge \emph{enters} the crossing $j+m$ and \emph{leaves} the crossing $j-1$. Each crossing $j$ has
two incoming edges $\mathrm{in}_L(j)\ni(\nu_j,j)$, $\mathrm{in}_R(j)\ni(\nu_j+1,j)$ and two outgoing edges
$\mathrm{out}_L(j)\ni(\nu_j,j+1)$, $\mathrm{out}_R(j)\ni(\nu_j+1,j+1)$; there are $2c$ edges.
We denote by $\hate$ the edge containing the node $(1,1)$ (the place where the closed braid is cut open).

It is convenient to name the four edges at a crossing according to the sign:
\begin{equation}\label{eq.ou}
 (o_j,u_j,o'_j,u'_j)=\begin{cases}
   (\mathrm{in}_L,\mathrm{in}_R,\mathrm{out}_R,\mathrm{out}_L)(j)&\text{if }\ve_j=+1,\\
   (\mathrm{in}_R,\mathrm{in}_L,\mathrm{out}_L,\mathrm{out}_R)(j)&\text{if }\ve_j=-1.
 \end{cases}
\end{equation}
If $\sigma_i$ is drawn as the positive crossing in which the strand going from the lower left to the upper right
is the over-strand, then $o_j,u_j$ are the incoming over- and under-edges and $o'_j,u'_j$ the outgoing ones.
(This is only a name; nothing below depends on the picture.)

\begin{definition}\label{def.T}
The \emph{edge matrices} $T^a(x)$ and $T^b(y)$ have rows and columns indexed by the edges, and
\begin{align*}
   T^a(x)&=\sum_{j=1}^c\Bigl(E_{o'_jo_j}+(1-x_j)\,E_{u'_jo_j}+x_j\,E_{u'_ju_j}\Bigr),\\
   T^b(y)&=\sum_{j=1}^c\Bigl(y_j\,E_{o'_jo_j}+(1-y_j)\,E_{u'_jo_j}+E_{u'_ju_j}\Bigr),
\end{align*}
where $E_{e'e}$ is the matrix unit with entry $1$ in row $e'$ and column $e$.
\end{definition}

In words: at each crossing a fermion of flavor $a$ on the edge $o$ continues on $o'$ with weight $1$ or drops
to $u'$ with weight $1-x_j$, and one on $u$ continues on $u'$ with weight $x_j$; for flavor $b$ the weights are
$y_j$, $1-y_j$ and $1$. For $y_j=t^{\ve_j}$ the rules for the flavor $b$ are the familiar ``random walk'' rules for
the Burau representation (a walker on the over-strand continues with weight $t^{\pm1}$ or drops to the under-strand
with weight $1-t^{\pm1}$), and those for the flavor $a$ differ from them by a diagonal change of basis. Note that,
with $e_f$ denoting the unit column vector of the edge $f$,
\begin{equation}\label{eq.rank1}
  T^a(x)=T^a(0)+\sum_jx_j\;e_{u'_j}\bigl(e_{u_j}-e_{o_j}\bigr)^{\T},\qquad
  T^b(y)=T^b(0)+\sum_jy_j\;\bigl(e_{o'_j}-e_{u'_j}\bigr)e_{o_j}^{\T}.
\end{equation}

\begin{proposition}[edge form]\label{prop.edge}
With the notation above,
\[
   P_w(x)=\det\bigl((1-T^a(x))_{\hate}\bigr),\qquad \tP_w(y)=\det\bigl((1-T^b(y))_{\hate}\bigr),
\]
where the subscript $\hate$ denotes deletion of the row and column of the edge $\hate$. Consequently
\[
   V_1(L;t_0,t_1)=(-1)^{\mu-1}\ \BE\Bigl[\det\bigl((1-T^a(x))_{\hate}\bigr)\,\det\bigl((1-T^b(y))_{\hate}\bigr)\Bigr].
\]
\end{proposition}
\begin{proof}
\emph{Step 1.} Let $\BT^\circ$ be $\BT(x)$ with the row and the column of the node $(1,1)$ replaced by zero. Then
$\det((1-\BT(x))_{\widehat{(1,1)}})=\det(1-\BT^\circ)$ (expand along the row $(1,1)$ of $1-\BT^\circ$, which is a
unit row). The matrix $\BT^\circ$ is of the form considered in Lemma~\ref{lem.cyclic}, with $C_1=B_1\Pi$,
$C_c=\Pi B_c$ ($C_1=\Pi B_1\Pi$ if $c=1$) and $C_j=B_j$ otherwise, where $\Pi=\diag(0,1,\dots,1)$. Hence
\[
  \det\bigl((1-\BT(x))_{\widehat{(1,1)}}\bigr)=\det\bigl(1-\Pi\,B(x)\,\Pi\bigr)=\det\bigl(1-B(x)'\bigr)=P_w(x),
\]
because $\Pi B(x)\Pi=0\oplus B(x)'$.

\emph{Step 2.} If $(p,j)$ is a passive node, then $\xi=(p,j)$ and $\zeta=(p,j+1)$ satisfy the hypothesis of
Lemma~\ref{lem.merge} for $Z=\BT(x)$. Merging $\zeta$ into $\xi$ produces a matrix indexed by
$\CN\smallsetminus\{\zeta\}$ in which $\xi$ has inherited the column of $\zeta$; if $\zeta$ was itself passive, the pair
$(\xi,(p,j+2))$ again satisfies
the hypothesis, and the hypothesis for all other passive nodes is unaffected. Applying Lemma~\ref{lem.merge}
successively to all passive nodes, each edge (a chain of nodes) is merged into its first node, which inherits the
column of its last, active, node. The resulting matrix is indexed by the edges, and its entry in row $e'$ and
column $e$ is $\BT(x)[(\text{first node of }e'),(\text{last node of }e)]$, i.e.\ the sum over the crossings $j$
which $e$ enters and $e'$ leaves of the corresponding entry of the block $\rho^{\ve_j}(x_j)$ of $B_j(x_j)$, whose
rows correspond to $(\mathrm{out}_L,\mathrm{out}_R)$ and columns to $(\mathrm{in}_L,\mathrm{in}_R)$. Comparing
\[
  \rho(x)=\bordermatrix{&\scriptstyle o&\scriptstyle u\cr \scriptstyle u'&1-x&x\cr \scriptstyle o'&1&0},\qquad
  \rho^-(x)=\bordermatrix{&\scriptstyle u&\scriptstyle o\cr \scriptstyle o'&0&1\cr \scriptstyle u'&x&1-x}
\]
with \eqref{eq.ou} and Definition~\ref{def.T}, this matrix is $T^a(x)$. By the last part of Lemma~\ref{lem.merge},
the deleted node $(1,1)$ becomes the deleted edge $\hate$. Thus
$\det((1-T^a(x))_{\hate})=\det((1-\BT(x))_{\widehat{(1,1)}})=P_w(x)$.

\emph{Step 3.} The same argument applies to $\widetilde\BT(y)$, whose blocks are the transposed matrices
\[
  \rho(y)^{\T}=\bordermatrix{&\scriptstyle o&\scriptstyle u\cr \scriptstyle u'&1-y&1\cr \scriptstyle o'&y&0},\qquad
  \rho^-(y)^{\T}=\bordermatrix{&\scriptstyle u&\scriptstyle o\cr \scriptstyle o'&0&y\cr \scriptstyle u'&1&1-y},
\]
and gives $T^b(y)$. The last assertion is Theorem~\ref{thm.B}.
\end{proof}

\subsection{The Berezin integral}
Let $\CB$ be the exterior algebra over $\BK$ on generators $\psi_k,\psi^+_k$, $k\in\mathcal K$, where $\mathcal K$ is a
finite set with $N$ elements. As in \cite{KauffmanSaleur}, $\psi^+_k$ is an independent generator, not a conjugate.
Elements of even degree commute with all elements of $\CB$. If $\chi\in\CB$ has no constant term, then
$\exp(\chi)=\sum_{m\ge0}\chi^m/m!$ is a finite sum, and $\exp(\chi+\chi')=\exp(\chi)\exp(\chi')$ if $\chi,\chi'$ are even.
The \emph{Berezin integral} is the linear form
\[
   \int\colon\CB\to\BK,\qquad \int\omega=\text{the coefficient of }\prod_{k\in\mathcal K}\psi_k\psi_k^+\text{ in }\omega
\]
(the factors $\psi_k\psi_k^+$ are even, so their order is irrelevant). This is the normalisation
\cite[(6.3)]{KauffmanSaleur}.

\begin{lemma}[Gaussian integral]\label{lem.gauss}
For every matrix $A=(A_{kl})_{k,l\in\mathcal K}$ with entries in $\BK$,
\[
   \int\exp\Bigl(\sum_{k,l\in\mathcal K}\psi_k\,A_{kl}\,\psi^+_l\Bigr)=\det A .
\]
\end{lemma}
\begin{proof}
Put $\th_k=\sum_lA_{kl}\psi_l^+$. The elements $\psi_k\th_k$ are even and $(\psi_k\th_k)^2=0$, so
$\exp(\sum_k\psi_k\th_k)=\prod_k(1+\psi_k\th_k)$, and its component of top degree $2N$ is
$\prod_k\psi_k\th_k$. Number $\mathcal K=\{1,\dots,N\}$. Then
\[
  \prod_{k=1}^N\psi_k\th_k=\sum_{l_1,\dots,l_N}A_{1l_1}\cdots A_{Nl_N}\ \psi_1\psi^+_{l_1}\psi_2\psi^+_{l_2}\cdots\psi_N\psi^+_{l_N},
\]
and a term vanishes unless $k\mapsto l_k$ is a permutation $\pi$. Moving all $\psi$'s to the left costs a sign
$s_N$ which does not depend on $\pi$, so
\begin{align*}
 \psi_1\psi^+_{\pi(1)}\cdots\psi_N\psi^+_{\pi(N)}&=s_N\,\psi_1\cdots\psi_N\,\psi^+_{\pi(1)}\cdots\psi^+_{\pi(N)}
   =s_N\,\sgn(\pi)\,\psi_1\cdots\psi_N\,\psi^+_{1}\cdots\psi^+_{N}\\
   &=\sgn(\pi)\,\psi_1\psi_1^+\cdots\psi_N\psi_N^+ .
\end{align*}
Hence the integral is $\sum_\pi\sgn(\pi)\prod_kA_{k\pi(k)}=\det A$.
\end{proof}

\begin{lemma}\label{lem.nil}
Let $p=(p_k)$ and $\ell=(\ell_k)$ be vectors indexed by $\mathcal K$ and
\[
   X=\Bigl(\sum_kp_k\psi_k\Bigr)\Bigl(\sum_{k'}\ell_{k'}\psi_{k'}^+\Bigr).
\]
Then $X$ is even, $X^2=0$, and $\exp(zX)=1+zX$ for every scalar $z$. Moreover, for every matrix $A$ indexed by
$\mathcal K$,
\[
   \sum_{k,k'}\psi_k\bigl(A-z\,p\ell^{\T}\bigr)_{kk'}\psi_{k'}^+=\sum_{k,k'}\psi_kA_{kk'}\psi^+_{k'}-zX .
\]
\end{lemma}
\begin{proof}
$X=\th\th'$ with $\th,\th'$ of degree one, and $\th\th'\th\th'=-\th^2\th'^2=0$. The last statement is clear.
\end{proof}

We now take two copies of this structure, one for each flavor. Let $\CB$ be the Grassmann algebra on the
generators
\[
   \psi^a_e,\ \psi^{a+}_e,\ \psi^b_e,\ \psi^{b+}_e,\qquad e\text{ an edge of the closed braid diagram},\ e\ne\hate \text{ the cut edge} ,
\]
and put $\psi^\sigma_{\hate}=\psi^{\sigma+}_{\hate}=0$ for $\sigma\in\{a,b\}$. Let $\bx_j=t_0^{\ve_j}$, $\by_j=t_1^{\ve_j}$,
$g_j=1-(t_0t_1)^{\ve_j}$ as in \eqref{eq.gj}, and define the \emph{free actions}
\begin{align*}
 \CS_a&=\sum_{e}\psi^a_e\psi^{a+}_e-\sum_{j=1}^c\Bigl(\psi^a_{o'_j}\psi^{a+}_{o_j}+(1-\bx_j)\,\psi^a_{u'_j}\psi^{a+}_{o_j}+\bx_j\,\psi^a_{u'_j}\psi^{a+}_{u_j}\Bigr),\\
 \CS_b&=\sum_{e}\psi^b_e\psi^{b+}_e-\sum_{j=1}^c\Bigl(\by_j\,\psi^b_{o'_j}\psi^{b+}_{o_j}+(1-\by_j)\,\psi^b_{u'_j}\psi^{b+}_{o_j}+\psi^b_{u'_j}\psi^{b+}_{u_j}\Bigr),
\end{align*}
that is, $\CS_\sigma=\sum_{e,e'}\psi^\sigma_e\bigl(1-T^\sigma\bigr)_{ee'}\psi^{\sigma+}_{e'}$ with the edge matrices of
Definition~\ref{def.T} evaluated at $x=\bx$, $y=\by$. Define also, for each crossing $j$,
\begin{equation}\label{eq.XY}
   X_j=\psi^a_{u'_j}\bigl(\psi^{a+}_{u_j}-\psi^{a+}_{o_j}\bigr),\qquad
   Y_j=\bigl(\psi^b_{o'_j}-\psi^b_{u'_j}\bigr)\psi^{b+}_{o_j}.
\end{equation}

\begin{theorem}\label{thm.C}
With the notation above,
\[
   V_1(L;t_0,t_1)=(-1)^{\mu-1}\int\exp\Bigl(\CS_a+\CS_b+\sum_{j=1}^cg_j\,X_jY_j\Bigr).
\]
Explicitly, the interaction is the
sum over the crossings of the quartic vertices
\[
   \bigl(1-(t_0t_1)^{\ve_j}\bigr)\ \psi^a_{u'_j}\bigl(\psi^{a+}_{u_j}-\psi^{a+}_{o_j}\bigr)\bigl(\psi^b_{o'_j}-\psi^b_{u'_j}\bigr)\psi^{b+}_{o_j}.
\]
\end{theorem}
\begin{proof}
Let $\d_1,\dots,\d_c,\d'_1,\dots,\d'_c$ be indeterminates. By \eqref{eq.rank1} and Lemma~\ref{lem.nil},
\[
  \sum_{e,e'}\psi^a_e\bigl(1-T^a(\bx+\d)\bigr)_{ee'}\psi^{a+}_{e'}=\CS_a-\sum_j\d_jX_j,\qquad
  \sum_{e,e'}\psi^b_e\bigl(1-T^b(\by+\d')\bigr)_{ee'}\psi^{b+}_{e'}=\CS_b-\sum_j\d'_jY_j .
\]
Apply Lemma~\ref{lem.gauss} to the index set $\mathcal K=\{a,b\}\times\{e\ne\hate\}$ and to the block diagonal matrix
with blocks $(1-T^a(\bx+\d))_{\hate}$ and $(1-T^b(\by+\d'))_{\hate}$ (the lemma and its proof hold over the
polynomial ring $\BK[\d,\d']$); by Proposition~\ref{prop.edge} its determinant is $P_w(\bx+\d)\,\tP_w(\by+\d')$. Since the $X_j,Y_j$ are even with square zero (Lemma~\ref{lem.nil}), we get
\[
   \int\exp(\CS_a+\CS_b)\ \prod_{j=1}^c(1-\d_jX_j)(1-\d'_jY_j)\;=\;P_w(\bx+\d)\,\tP_w(\by+\d') ,
\]
an identity between polynomials in $\d,\d'$. Comparing the coefficients of $\d^S\d'^S$ for $S\subseteq\{1,\dots,c\}$
gives $\int\exp(\CS_a+\CS_b)\prod_{j\in S}X_jY_j=a_Sb_S$ in the notation of Corollary~\ref{cor.expansion}. Multiplying
by $\prod_{j\in S}g_j$ and summing over $S$,
\[
  \int\exp(\CS_a+\CS_b)\prod_{j=1}^c(1+g_jX_jY_j)=\sum_S\Bigl(\prod_{j\in S}g_j\Bigr)a_Sb_S=(-1)^{\mu-1}\,V_1(L;t_0,t_1)
\]
by Corollary~\ref{cor.expansion}. Finally $\prod_j(1+g_jX_jY_j)=\exp(\sum_jg_jX_jY_j)$, because the $X_jY_j$ are even
and $(X_jY_j)^2=0$.
\end{proof}

\begin{remark}
Without the vertices, Theorem~\ref{thm.C} is the Gaussian integral of \cite[(6.13)]{KauffmanSaleur} taken once for each
flavor, with value $P_w(\bx)\tP_w(\by)$, up to the normalisation of the vertex weights (Kauffman and Saleur use a
symmetric normalisation of the Burau weights and a global prefactor). There are two further differences of
presentation. Kauffman and Saleur label their
Grassmann variables by the pairs (crossing, strand through this crossing) rather than by the edges. And they keep
the variables at the two ends of the edge where the diagram is cut open, omitting only the propagator of this edge,
so that their matrix has size $2c$ instead of $2c-1$. In the present notation this amounts to replacing
$(1-T^\sigma)_{\hate}$ by $1-T^\sigma-E_{\hate\hate}$, and
$\det(1-T^\sigma-E_{\hate\hate})=\det(1-T^\sigma)-\det((1-T^\sigma)_{\hate})=-\det((1-T^\sigma)_{\hate})$, because
$\det(1-T^\sigma)=0$ (the rows of $T^a$ and the columns of $T^b$ have sum $1$). So the two flavors are free
``Alexander fermions'' at $t_0$ and $t_1$, coupled only by a contact interaction at the crossings, of strength
$1-(t_0t_1)^{\pm1}$.
\end{remark}

\subsection{A state sum over pairs of spanning trees}\label{ssec.trees}
The matrix $1-T^a(x)$ has zero row sums and $1-T^b(y)$ has zero column sums. So the determinants of
Proposition~\ref{prop.edge} are of the kind computed by the matrix-tree theorem, which turns them into sums over
spanning trees. We first state this theorem in the form that we need, with a proof.

\begin{lemma}[matrix-tree theorem]\label{lem.mtt}
Let $Z$ be a square matrix with rows and columns indexed by a finite set $\CN$, with entries in a commutative ring,
and let $d\in\CN$. Assume that $\sum_{l\in\CN}Z[k,l]=1$ for every $k\ne d$. Call a map
$f\colon\CN\smallsetminus\{d\}\to\CN$ a \emph{tree rooted at $d$} if every $k\ne d$ is mapped to $d$ by some
iterate of $f$. Then, with the notation of Lemma~\ref{lem.merge},
\[
   \det\bigl((1-Z)_{\hat d}\bigr)=\sum_{f}\ \prod_{k\ne d}Z[k,f(k)] ,
\]
where the sum runs over all trees rooted at $d$.
\end{lemma}
\begin{proof}
Put $\CN^\circ=\CN\smallsetminus\{d\}$ and $A=(1-Z)_{\hat d}$. For $k\ne l$ in $\CN^\circ$ we have $A[k,l]=-Z[k,l]$,
and by the assumption on the row sums
\[
   A[k,k]=1-Z[k,k]=\sum_{l\in\CN,\ l\ne k}Z[k,l]\qquad(k\in\CN^\circ).
\]
Expand $\det A=\sum_\sigma\sgn(\sigma)\prod_{k\in\CN^\circ}A[k,\sigma(k)]$ over the permutations $\sigma$ of
$\CN^\circ$. A cycle $(k_1\,k_2\,\cdots\,k_m)$ of $\sigma$ of length $m\ge2$ contributes the sign $(-1)^{m-1}$ to
$\sgn(\sigma)$ and the factor $(-1)^m\prod_iZ[k_i,\sigma(k_i)]$ to the product, in total $-\prod_iZ[k_i,\sigma(k_i)]$.
For each fixed point $k$ of $\sigma$ we replace $A[k,k]$ by the sum above and expand. The result is
\[
   \det A=\sum_{(f,C)}(-1)^{|C|}\prod_{k\in\CN^\circ}Z[k,f(k)] ,
\]
where $f$ runs over the maps $\CN^\circ\to\CN$ without fixed points and $C$ runs over the sets of cycles of $f$. Here a
\emph{cycle} of $f$ is a subset $\{k_1,\dots,k_m\}$ of $\CN^\circ$ with $m\ge2$ and $f(k_i)=k_{i+1}$ for all $i$
(indices modulo $m$); distinct cycles of $f$ are disjoint. Indeed, a permutation $\sigma$ together with the choices
made at its fixed points gives the pair $(f,C)$ in which $f$ is $\sigma$ on the cycles of $\sigma$ of length $\ge2$
and the chosen index at the fixed points, and $C$ is the set of these cycles; conversely $(f,C)$ determines $\sigma$,
which is $f$ on the union of the cycles in $C$ and the identity elsewhere, and the choices.

For a fixed $f$ the sum $\sum_C(-1)^{|C|}$ over all sets $C$ of cycles of $f$ is $1$ if $f$ has no cycle and $0$
otherwise. Finally, a map $f\colon\CN^\circ\to\CN$ without fixed points and without cycles is a tree rooted at $d$:
for $k\in\CN^\circ$ the sequence $k,f(k),f(f(k)),\dots$ cannot stay in the finite set $\CN^\circ$ for ever, because
its first repetition would produce a fixed point or a cycle; so it reaches $d$. Conversely a tree rooted at $d$ has
neither fixed points nor cycles.
\end{proof}

We apply this to the edge matrices of Definition~\ref{def.T}. Recall that every edge leaves exactly one crossing
and enters exactly one crossing.

\begin{definition}\label{def.trees}
An \emph{$a$-tree} is a map $\phi$ from the set of edges different from $\hate$ to the set of all edges such that
\[
   \phi(o'_j)=o_j\qquad\text{and}\qquad\phi(u'_j)\in\{u_j,o_j\}\qquad\text{for all crossings }j
\]
(as far as $o'_j$, resp.\ $u'_j$, is different from $\hate$), and such that every edge different from $\hate$ is mapped
to $\hate$ by some iterate of $\phi$. A \emph{$b$-tree} is a map $\psi$ from the set of edges different from $\hate$
to the set of all edges such that
\[
   \psi(u_j)=u'_j\qquad\text{and}\qquad\psi(o_j)\in\{o'_j,u'_j\}\qquad\text{for all crossings }j
\]
(as far as $u_j$, resp.\ $o_j$, is different from $\hate$), with the same condition on the iterates. For an $a$-tree
$\phi$, a $b$-tree $\psi$ and a crossing $j$ we put
\[
  \vartheta_j(\phi)=\begin{cases}0&\text{if }\phi(u'_j)=u_j,\\ 1&\text{if }\phi(u'_j)=o_j,\\ \emptyset&\text{if }u'_j=\hate,\end{cases}
  \qquad\qquad
  \vartheta_j(\psi)=\begin{cases}0&\text{if }\psi(o_j)=o'_j,\\ 1&\text{if }\psi(o_j)=u'_j,\\ \emptyset&\text{if }o_j=\hate.\end{cases}
\]
\end{definition}

In words: an $a$-tree follows the strands of the diagram backwards. From an outgoing over-edge it goes back along
the over-strand; from an outgoing under-edge it either goes back along the under-strand ($\vartheta_j=0$) or switches
to the over-strand ($\vartheta_j=1$). A $b$-tree follows the strands forwards. From an incoming under-edge it
continues along the under-strand; from an incoming over-edge it either continues along the over-strand
($\vartheta_j=0$) or switches to the under-strand ($\vartheta_j=1$). In both cases the choices at the $c$ crossings
must lead every edge to the edge $\hate$ where the diagram is cut.

Let $p_0(z)=z$, $p_1(z)=1-z$ and $p_\emptyset(z)=1$, and recall the notation $\bx_j=t_0^{\ve_j}$, $\by_j=t_1^{\ve_j}$,
$g_j=1-\bx_j\by_j$ of \eqref{eq.gj}.

\begin{proposition}[pairs of spanning trees]\label{prop.trees}
\begin{enumerate}\renewcommand{\theenumi}{\roman{enumi}}
\item With $\phi$ running over the $a$-trees and $\psi$ over the $b$-trees,
      \[
         P_w(x)=\sum_{\phi}\ \prod_{j=1}^cp_{\vartheta_j(\phi)}(x_j),\qquad
         \tP_w(y)=\sum_{\psi}\ \prod_{j=1}^cp_{\vartheta_j(\psi)}(y_j).
      \]
\item With the sum running over all pairs of an $a$-tree $\phi$ and a $b$-tree $\psi$,
      \[
         V_1(L;t_0,t_1)=(-1)^{\mu-1}\sum_{(\phi,\psi)}\ \prod_{j=1}^cW_j\bigl(\vartheta_j(\phi),\vartheta_j(\psi)\bigr),
      \]
      where the local weights
      $W_j(\vartheta,\vartheta')=\BE_{\bx_j,\by_j}\bigl[p_{\vartheta}(x)\,p_{\vartheta'}(y)\bigr]$ are given by
      \[
        \renewcommand{\arraystretch}{1.2}
        \begin{array}{c|ccc}
          W_j(\vartheta,\vartheta') & \vartheta'=0 & \vartheta'=1 & \vartheta'=\emptyset\\ \hline
          \vartheta=0 & 1 & \bx_j-1 & \bx_j\\
          \vartheta=1 & \by_j-1 & 2-\bx_j-\by_j & 1-\bx_j\\
          \vartheta=\emptyset & \by_j & 1-\by_j & 1
        \end{array}
      \]
\item For $\vartheta,\vartheta'\in\{0,1\}$,
      \[
         W_j(\vartheta,\vartheta')=p_\vartheta(\bx_j)\,p_{\vartheta'}(\by_j)+(-1)^{\vartheta+\vartheta'}g_j ,
      \]
      and $W_j(\vartheta,\vartheta')=p_\vartheta(\bx_j)\,p_{\vartheta'}(\by_j)$ if $\vartheta=\emptyset$ or
      $\vartheta'=\emptyset$.
\end{enumerate}
\end{proposition}
\begin{proof}
(i) Every edge is the edge $o'_j$ or the edge $u'_j$ of exactly one crossing $j$. By Definition~\ref{def.T} the row
$o'_j$ of $T^a(x)$ has the single non-zero entry $1$ in the column $o_j$, and the row $u'_j$ has the entries $1-x_j$
in the column $o_j$ and $x_j$ in the column $u_j$. So the rows of $T^a(x)$ have sum $1$, and
Lemma~\ref{lem.mtt}, applied to $Z=T^a(x)$ and $d=\hate$, gives
$\det((1-T^a(x))_{\hate})=\sum_f\prod_{e\ne\hate}T^a(x)[e,f(e)]$. The trees $f$ with a non-zero contribution are
exactly the $a$-trees $\phi$, and the contribution of $\phi$ is $\prod_jp_{\vartheta_j(\phi)}(x_j)$. By
Proposition~\ref{prop.edge} the left-hand side is $P_w(x)$. In the same way, every edge is the edge $o_j$ or the edge
$u_j$ of exactly one crossing; the column $u_j$ of $T^b(y)$ has the single non-zero entry $1$ in the row $u'_j$, and
the column $o_j$ has the entries $y_j$ in the row $o'_j$ and $1-y_j$ in the row $u'_j$. So the transposed matrix
$Z=T^b(y)^{\T}$ satisfies the hypothesis of Lemma~\ref{lem.mtt}, the trees with a non-zero contribution are the
$b$-trees, and $\tP_w(y)=\det((1-T^b(y))_{\hate})=\det((1-Z)_{\hate})=\sum_\psi\prod_jp_{\vartheta_j(\psi)}(y_j)$.

(ii) By (i), $P_w(x)\tP_w(y)=\sum_{(\phi,\psi)}\prod_jp_{\vartheta_j(\phi)}(x_j)\,p_{\vartheta_j(\psi)}(y_j)$, and the
$j$-th factor is a polynomial of degree $\le1$ in $x_j$ and in $y_j$ which involves no other variable. By
\eqref{eq.Emult}, $\BE$ of each product is $\prod_jW_j(\vartheta_j(\phi),\vartheta_j(\psi))$, and the claim follows
from Theorem~\ref{thm.B}. The table is obtained from \eqref{eq.E1}: for instance
$\BE_{u,v}[(1-x)y]=v-1$ and $\BE_{u,v}[(1-x)(1-y)]=2-u-v$.

(iii) This is \eqref{eq.E1} again: $\BE_{u,v}[f]=f(u,v)+(1-uv)f_{11}$, and the coefficient of $xy$ in
$p_\vartheta(x)p_{\vartheta'}(y)$ is $(-1)^{\vartheta+\vartheta'}$ for $\vartheta,\vartheta'\in\{0,1\}$ and $0$ if one of
them is $\emptyset$.
\end{proof}

\begin{remark}\label{rem.trees}
By (iii), $V_1$ is the partition function of two spanning trees of the diagram which interact only at the crossings
where both of them have a choice: there the weight $p_\vartheta(\bx_j)p_{\vartheta'}(\by_j)$ of the two independent
trees is changed by $+g_j$ if they make the same choice and by $-g_j$ if they make different choices. Without
this term, i.e.\ for $t_0t_1=1$, the state sum factors into $P_w(\bx)\tP_w(\by)$ by (i). For $x_j=t^{\ve_j}$ the
weights $x_j$, $1-x_j$ are those of the random walk on the diagram which underlies the Burau representation, and (i)
is the corresponding expansion of the Alexander polynomial. These trees are in general more numerous than the
Kauffman states of the diagram: the closure of $\sigma_1^5$ has $11$ $a$-trees, whereas its diagram has $5$ Kauffman
states.
\end{remark}

\begin{example}[trefoil]\label{ex.trees}
Let $w=\sigma_1^3$, as in Example~\ref{ex.trefoil}. The six edges are the nodes $(p,j)$ with $p\in\{1,2\}$ and
$j\in\BZ/3$, and
\[
   (o_j,u_j,o'_j,u'_j)=\bigl((1,j),\,(2,j),\,(2,j+1),\,(1,j+1)\bigr),\qquad \hate=(1,1)=o_1=u'_3 .
\]
So $\vartheta_3(\phi)=\emptyset$ for every $a$-tree and $\vartheta_1(\psi)=\emptyset$ for every $b$-tree. There are
three $a$-trees, with $(\vartheta_1,\vartheta_2)=(0,0)$, $(1,0)$, $(1,1)$: the fourth choice $(0,1)$ gives the cycle
$(1,2)\mapsto(2,1)\mapsto(1,3)\mapsto(1,2)$. Likewise there are three $b$-trees, with
$(\vartheta_2,\vartheta_3)=(0,0)$, $(0,1)$, $(1,1)$. Hence
\[
   P_w=x_1x_2+(1-x_1)x_2+(1-x_1)(1-x_2),\qquad \tP_w=y_2y_3+y_2(1-y_3)+(1-y_2)(1-y_3),
\]
in agreement with Example~\ref{ex.trefoil}. In the state sum (ii), the $a$-trees $(0,0)$ and $(1,0)$ contribute
$t_0\,\Sigma$ and $(1-t_0)\,\Sigma$ with $\Sigma=t_1+(1-t_1)+(t_0-1)(1-t_1)$, and the $a$-tree $(1,1)$ contributes
$(1-t_0)\bigl[(t_1-1)t_1+(t_1-1)(1-t_1)+(2-t_0-t_1)(1-t_1)\bigr]$. The sum of the nine terms is
\[
   V_1(3_1)=1-(1-t_0)(1-t_1)+(1-t_0)(1-t_1)(1-t_0-t_1)=1-(t_0+t_1)(1-t_0)(1-t_1).
\]
\end{example}

\section{Arbitrary diagrams}\label{sec.diagrams}

In Sections~\ref{sec.theorem} and~\ref{sec.lagrangian} the link was the closure of a braid, cut along its first
strand. We now show that the edge formula of Proposition~\ref{prop.edge} holds for every connected diagram of the
link and for every edge of this diagram (Theorem~\ref{thm.diagrams}). The proof has three ingredients: the case of a
closed braid cut along an arbitrary strand (Section~\ref{ssec.strandp}), the invariance of the right-hand side under
Reidemeister moves of type~II (Section~\ref{ssec.RII}), and the theorem of Yamada and Vogel, according to which such
moves transform every diagram into a closed braid diagram (Section~\ref{ssec.thmdiagrams}). Besides
Theorem~\ref{thm.GHKST}, two results are used here: property (P$_2$) of the $R$-matrix
from \cite[Lemma~3.2]{GHKST}, from which we deduce Theorem~\ref{thm.GHKSTp}, and the theorem of Yamada and Vogel
(Theorem~\ref{thm.vogel}).

\subsection{Diagrams and their edge matrices}\label{ssec.diagdef}
In this section and the next one, a \emph{diagram} is an oriented link diagram $D$ on the oriented sphere $S^2$
which is connected and has $c\ge1$ crossings, numbered $1,\dots,c$. The sign $\ve_j$ of the crossing $j$ is the usual
one: the crossing is positive if, when it is drawn with both strands pointing upwards, the over-strand runs from
the lower left to the upper right, and it is negative otherwise. The crossings cut $D$ into $2c$ oriented arcs, the
\emph{edges} of $D$; every edge leaves one crossing and enters one crossing. At the crossing $j$ we denote by
$o_j,u_j$ the incoming edges of the over-strand and of the under-strand, and by $o'_j,u'_j$ the outgoing ones. The \emph{combinatorial data} of $D$ are the signs $\ve_j$ and the edges $o_j,u_j,o'_j,u'_j$; they
are not changed by an orientation-preserving homeomorphism of $S^2$. The \emph{edge matrices} of $D$ are defined
by the formulas of Definition~\ref{def.T}:
\begin{align*}
   T^a_D(x)&=\sum_{j=1}^c\Bigl(E_{o'_jo_j}+(1-x_j)\,E_{u'_jo_j}+x_j\,E_{u'_ju_j}\Bigr),\\
   T^b_D(y)&=\sum_{j=1}^c\Bigl(y_j\,E_{o'_jo_j}+(1-y_j)\,E_{u'_jo_j}+E_{u'_ju_j}\Bigr).
\end{align*}
For an edge $\hate$ of $D$ we put
\[
   P_{D,\hate}(x)=\det\bigl((1-T^a_D(x))_{\hate}\bigr),\qquad \tP_{D,\hate}(y)=\det\bigl((1-T^b_D(y))_{\hate}\bigr).
\]
As in \eqref{eq.rank1}, $T^a_D(x)$ is the sum of $T^a_D(0)$ and of the matrices $x_j\,e_{u'_j}(e_{u_j}-e_{o_j})^{\T}$ of
rank one, and similarly for $T^b_D(y)$. By Lemma~\ref{lem.rankone}, $P_{D,\hate}$ has degree $\le1$ in each $x_j$ and
$\tP_{D,\hate}$ has degree $\le1$ in each $y_j$. Let $\BE_D$ be the linear map \eqref{eq.E} for the crossings of
$D$: it is defined on the polynomials of degree $\le1$ in each of the variables $x_j,y_j$, and
$\BE_D[x^Sy^{S'}]=\prod_{j\in S\smallsetminus S'}t_0^{\ve_j}\prod_{j\in S'\smallsetminus S}t_1^{\ve_j}$. If $D$ is a diagram
of a link with $\mu$ components, we put
\begin{equation}\label{eq.Upsilon}
   \Upsilon(D,\hate)=(-1)^{\mu-1}\ \BE_D\bigl[P_{D,\hate}(x)\,\tP_{D,\hate}(y)\bigr].
\end{equation}
The polynomials $P_{D,\hate}$, $\tP_{D,\hate}$ and the element $\Upsilon(D,\hate)$ of $\BK$ depend only on the
combinatorial data of $D$ and on $\hate$. Indeed, the map which sends $o_j$ to $o'_j$ and $u_j$ to $u'_j$ for all $j$
is a permutation of the set of edges (the edge which follows a given edge along its strand), and $\mu$ is the
number of its cycles.

Let $w$ be a braid word in which every position is moved by at least one letter, as in
Section~\ref{sec.lagrangian}. Its \emph{closure diagram} $D_w$ is the standard diagram of the closed braid
$\widehat w$, drawn in the plane and viewed on $S^2$: the strands of the braid run upwards, the positions
$1,\dots,n$ are numbered from the left to the right, the letters $s_1,\dots,s_c$ are met in this order from the bottom
to the top, a letter $\sigma_i$ is a positive crossing and a letter $\sigma_i^{-1}$ a negative crossing between the strands
in the positions $i$ and $i+1$, and the $n$ strands are closed by disjoint arcs on the right of the braid. It is a
diagram in the present sense: its crossings are the letters of $w$, its edges are the edges of
Section~\ref{ssec.edges}, and the edges $o_j,u_j,o'_j,u'_j$ are those of \eqref{eq.ou}, because the over-strand of a
positive crossing comes in from the left and the over-strand of a negative crossing comes in from the right. Thus $T^a_{D_w}$, $T^b_{D_w}$ are the matrices of
Definition~\ref{def.T}, and Proposition~\ref{prop.edge} says that $V_1(\widehat w)=\Upsilon(D_w,\hate)$ when $\hate$ is
the edge which contains the node $(1,1)$.

\subsection{Closed braids cut along an arbitrary strand}\label{ssec.strandp}
Let $w=s_1\cdots s_c$ be a braid word on $n$ strands and let $1\le p\le n$. With the notation of
Definition~\ref{def.V} and Remark~\ref{rem.convention}, we put
\[
   V^{(p)}(w)\;=\;\sum_{\substack{I=(m_1,\dots,m_n)\\ m_p=1}}\ \Bigl(\prod_{k\ne p}h_{m_k}\Bigr)\,
       \bigl\langle e_I,\ \varrho(w)\,e_I\bigr\rangle ,
\]
so that $V^{(1)}(w)=V(w)$. This is the $(e_1,e_1)$ matrix entry of the partial trace, over all the tensor factors
except the $p$-th one, of the operator $(h^{\otimes(p-1)}\otimes\id_W\otimes h^{\otimes(n-p)})\circ\varrho(w)$.

\begin{theorem}[{\cite[Section~1.1, Remark~2.2 and Lemma~3.2]{GHKST}}]\label{thm.GHKSTp}
For $r=\pm1$ and $1\le p\le n$, the partial trace over all the tensor factors except the $p$-th one of the operator
$(h^{\otimes(p-1)}\otimes\id_W\otimes h^{\otimes(n-p)})\circ\varrho_r(w)$ is equal to $V_1(\widehat w)\,\id_W$. In particular
$V^{(p)}(w)=V_1(\widehat w)$.
\end{theorem}

For $p=1$ this is Theorem~\ref{thm.GHKST}. Let us explain how the general case follows from \cite{GHKST}. The
enhanced $R$-matrix $(R_{\pm1},h)$ defines an operator invariant $F_T$ of oriented tangles $T$, which may have closed
components; it is computed from a diagram of $T$ with the cups and caps of
\cite[Remark~2.2]{GHKST}.\footnote{In the last formula of \cite[Remark~2.2]{GHKST} we read
$\overrightarrow{\mathrm{coev}}_V(1)=\sum_ie_i^*\otimes h^{-1}(e_i)$, which gives the formula for
the closure on the left used in the proof of \cite[Lemma~3.2]{GHKST}.} Let $T$ be a tangle diagram whose boundary points are all oriented upwards, and let
$f=F_T$. With these cups and caps, closing the last strand of $T$ on the right gives the partial trace over the
last tensor factor of $(\id\otimes h)\circ f$, and closing the first strand of $T$ on the left gives the partial trace
over the first tensor factor of $f\circ(h^{-1}\otimes\id)$; for instance, for $f\in\End(W)$ the closure on the left is
$\sum_ie_i^*\bigl(f(h^{-1}e_i)\bigr)=\tr(f\circ h^{-1})$. These are the two formulas which are used in the proof of
\cite[Lemma~3.2]{GHKST}, and the first one gives \cite[Remark~2.3]{GHKST}. A partial trace over a tensor factor does
not change when an operator which acts on this factor only is moved from the right of $f$ to the left of $f$, and
$h^{-1}=h$. So the closure on the left also gives the partial trace over the first tensor factor of
$(h\otimes\id)\circ f$.

Let $\beta=s_c\cdots s_1\in B_n$ be the braid of Section~\ref{sec.setup}, whose operator is
$\rho_{R_r}(\beta)=\varrho_r(w)$ and whose closure is $\widehat w$. For $1\le p\le n$ let $\Xi_p$ be the $(1,1)$-tangle
obtained from the diagram of $\beta$ by closing the strands in the positions $1,2,\dots,p-1$ on the left, in this
order, and the strands in the positions $n,n-1,\dots,p+1$ on the right, in this order; thus the closing arc of the
position $1$ is the innermost one on the left, and the closing arc of the position $n$ is the innermost one on the
right. By the above, $F_{\Xi_p}$ is the partial trace of the theorem. For $1\le p\le n-1$ let $\Xi^{(p)}$ be the
$(2,2)$-tangle obtained in the same way by closing the strands in the positions $1,\dots,p-1$ on the left and the
strands in the positions $n,\dots,p+2$ on the right. Its boundary points are oriented upwards. Closing its right
strand on the right gives $\Xi_p$, and closing its left strand on the left gives $\Xi_{p+1}$; these are the two
closures $T_1$ and $T_2$ of the $(2,2)$-tangle $T=\Xi^{(p)}$ in \cite[Figure~1]{GHKST}. Now \cite[Lemma~3.2]{GHKST}
states that the enhanced $R$-matrix $(R_{\pm1},h)$ has the property (P$_2$) of \cite[Section~1.1]{GHKST}:
$F_{T_1}=F_{T_2}$ for every $(2,2)$-tangle $T$ with upward oriented boundary. Hence $F_{\Xi_{p+1}}=F_{\Xi_p}$ for
$1\le p\le n-1$. Therefore $F_{\Xi_p}=F_{\Xi_1}$ for all $p$, and $F_{\Xi_1}=V_1(\widehat w)\,\id_W$ by
Theorem~\ref{thm.GHKST}.

For an $n\times n$ matrix $M$ let $M^{(p)}$ be the matrix $M$ with its $p$-th row and its $p$-th column removed; thus
$M^{(1)}=M'$.

\begin{proposition}\label{prop.strandp}
Let $w$ be a braid word on $n$ strands with $c$ crossings and let $1\le p\le n$. Then
\[
   V^{(p)}(w)=(-1)^{n-1+c}\ \BE\Bigl[\det\bigl(1-B(x)^{(p)}\bigr)\,\det\bigl(1-\tB(y)^{(p)}\bigr)\Bigr].
\]
Consequently, if $L$ is the closure of $w$, then for every $p$
\[
   V_1(L;t_0,t_1)=(-1)^{\mu-1}\ \BE\Bigl[\det\bigl(1-B(x)^{(p)}\bigr)\,\det\bigl(1-\tB(y)^{(p)}\bigr)\Bigr].
\]
\end{proposition}
\begin{proof}
Steps~1 and~2 of the proof of Theorem~\ref{thm.B} do not involve the first strand, and they give \eqref{eq.step2}.
In Step~3 we replace the set $Q$ by $Q_p=\{1,\dots,2n\}\smallsetminus\{2p-1,2p\}$, the set of the indices of the
basis vectors of $U_n$ other than $a_p$ and $b_p$. By \eqref{eq.Phi}, the vectors $e_I$ with $m_p=1$ correspond
under $\Phi_n$ to the $u_S$ with $S\subseteq Q_p$, and $\prod_{k\ne p}h_{m_k}=(-1)^{n-1}(-1)^{|S|}$ by
\eqref{eq.hparity}. Conjugation by the diagonal operator $K_n$ does not change diagonal matrix entries. Hence, by
\eqref{eq.step2}, Lemma~\ref{lem.compress} and the linearity of $\BE$,
\[
   V^{(p)}(w)=(-1)^{n-1+c}\sum_{S\subseteq Q_p}(-1)^{|S|}\,\BE\bigl[\langle u_S,\ext(\BB)u_S\rangle\bigr]
       =(-1)^{n-1+c}\,\BE\bigl[\det\bigl(1-\BB[Q_p|Q_p]\bigr)\bigr],
\]
where $\BB=B(x)\oplus\tB(y)$. The matrix $\BB[Q_p|Q_p]$ is the endomorphism of the span of the $a_k,b_k$ with $k\ne p$
which acts by $B(x)^{(p)}$ on the $a_k$ and by $\tB(y)^{(p)}$ on the $b_k$, so
$\det(1-\BB[Q_p|Q_p])=\det(1-B(x)^{(p)})\det(1-\tB(y)^{(p)})$. The last statement follows from
Theorem~\ref{thm.GHKSTp} and from $(-1)^{n-1+c}=(-1)^{\mu-1}$, see \eqref{eq.VV1}.
\end{proof}

\begin{proposition}\label{prop.braidp}
Let $w$ be a braid word in which every position is moved by at least one letter, let $L$ be its closure and $D_w$
its closure diagram. Then $V_1(L;t_0,t_1)=\Upsilon(D_w,\hate)$ for every edge $\hate$ of $D_w$.
\end{proposition}
\begin{proof}
The edge $\hate$ contains a node $(p,j)$. The word $s_js_{j+1}\cdots s_cs_1\cdots s_{j-1}$ has the same closure diagram
as $w$, with the same edges and with the crossings renumbered, and the node $(p,j)$ of $w$ is the node $(p,1)$ of
this word. Since $\Upsilon(D_w,\hate)$ does not depend on the numbering of the crossings, we may assume that $j=1$.

We repeat the proof of Proposition~\ref{prop.edge} with the node $(p,1)$ in place of the node $(1,1)$. In Step~1
the matrix $\Pi$ is replaced by the diagonal matrix $\Pi_p$ with the entry $0$ in the position $p$ and the entries
$1$ elsewhere, and one obtains
\[
   \det\bigl((1-\BT(x))_{\widehat{(p,1)}}\bigr)=\det\bigl(1-\Pi_p\,B(x)\,\Pi_p\bigr)=\det\bigl(1-B(x)^{(p)}\bigr).
\]
Step~2 is unchanged: the passive nodes are merged by Lemma~\ref{lem.merge}, and the deleted node $(p,1)$ becomes the
deleted edge $\hate$. Hence $P_{D_w,\hate}(x)=\det(1-B(x)^{(p)})$, and in the same way
$\tP_{D_w,\hate}(y)=\det(1-\tB(y)^{(p)})$. The claim now follows from Proposition~\ref{prop.strandp}.
\end{proof}

\subsection{Reidemeister moves of type II}\label{ssec.RII}
We first record the elimination of a node from a determinant of the form $\det(1-Z)$; Lemma~\ref{lem.merge} is a
special case of it.

\begin{lemma}[elimination of a node]\label{lem.elim}
Let $Z$ be a square matrix with rows and columns indexed by a finite set $\CN$, with entries in a commutative ring,
and let $\zeta\in\CN$ be such that $Z[\zeta,\zeta]=0$. Let $Z^\natural$ be the matrix indexed by
$\CN\smallsetminus\{\zeta\}$ with
\[
   Z^\natural[k,l]=Z[k,l]+Z[k,\zeta]\,Z[\zeta,l].
\]
Then $\det(1-Z)=\det(1-Z^\natural)$, and $\det\bigl((1-Z)_{\hat d}\bigr)=\det\bigl((1-Z^\natural)_{\hat d}\bigr)$ for every
$d\in\CN$ different from $\zeta$.
\end{lemma}
\begin{proof}
Order $\CN$ so that $\zeta$ comes last. Then $1-Z=\begin{psmallmatrix}A&B\\D&1\end{psmallmatrix}$, where $A$ is the
matrix $1-Z$ restricted to $\CN\smallsetminus\{\zeta\}$, $B$ is the column with the entries $-Z[k,\zeta]$ and $D$ is
the row with the entries $-Z[\zeta,l]$. By the identity
$\det\begin{psmallmatrix}A&B\\D&1\end{psmallmatrix}=\det(A-BD)$ used in the proof of Lemma~\ref{lem.cyclic}, and since
$A-BD=1-Z^\natural$, we get the first claim. The second claim is the first one for the matrix obtained from $Z$ by
removing the row and the column of $d$, because this operation commutes with the passage from $Z$ to $Z^\natural$.
\end{proof}

Next we describe a Reidemeister move of type~II in terms of the combinatorial data. Let $D$ be a diagram with $c$
crossings, let $A\ne B$ be two edges of $D$ and let $\ve\in\{1,-1\}$. Consider a diagram $D^+$ with $c+2$ crossings
whose edges are the edges of $D$ other than $A$ and $B$ and six further edges $A_1,A_2,A_3,B_1,B_2,B_3$, with the
following combinatorial data. At the crossings $1,\dots,c$ the data are those of $D$, except that $A$ is replaced
by $A_1$ where it is an outgoing edge and by $A_3$ where it is an incoming edge, and likewise for $B$. The two new
crossings have the signs $\ve_{c+1}=\ve$ and $\ve_{c+2}=-\ve$, and either
\begin{align*}
  (o,u,o',u')_{c+1}&=(A_1,B_1,A_2,B_2), & (o,u,o',u')_{c+2}&=(A_2,B_2,A_3,B_3) &&\text{(parallel case), or}\\
  (o,u,o',u')_{c+1}&=(A_1,B_2,A_2,B_3), & (o,u,o',u')_{c+2}&=(A_2,B_1,A_3,B_2) &&\text{(antiparallel case).}
\end{align*}
We then say that $D^+$ is obtained from $D$ by a \emph{move of type~II in which $A$ passes over $B$}. This is what
happens when a piece of the edge $A$ is pushed over a piece of the edge $B$ across a face of $D$ which is adjacent
to both: each of the two edges is cut into three pieces by the two new crossings, the edge $A$ is the over-strand at
both of them, and they have opposite signs; we number them in the order in which they are met along $A$, and they
are met in the same order along $B$ if the two pieces are parallel, and in the opposite order if they are
antiparallel. The link is not changed by such a move. (In this situation the sign $\ve$ is determined by the
picture. We do not need to know it: Lemma~\ref{lem.RII} holds for both values of $\ve$.)

\begin{lemma}\label{lem.RII}
Let $D^+$ be obtained from $D$ by a move of type~II in which $A$ passes over $B$. Let $\hate$ be an edge of $D$ and
let $\hate^+$ be the edge of $D^+$ defined by $\hate^+=\hate$ if $\hate\notin\{A,B\}$, $\hate^+=A_1$ if $\hate=A$ and
$\hate^+=B_1$ if $\hate=B$. Then
\[
    \Upsilon(D^+,\hate^+)=\Upsilon(D,\hate).
\]
\end{lemma}
\begin{proof}
We write $A,B$ for the edges $A_1,B_1$ of $D^+$, so that the edges of $D^+$ are the edges of $D$ and the four edges
$A_2,B_2,A_3,B_3$, and $\hate^+=\hate$. We write $x',x''$ for $x_{c+1},x_{c+2}$ and $y',y''$ for $y_{c+1},y_{c+2}$.

\smallskip\emph{The matrix $Z=T^a_{D^+}(x)$.} Let $k$ be an edge of $D$. The row $k$ of $Z$ is the row $k$ of
$T^a_D(x)$ in which the entries of the columns $A$ and $B$ have been moved to the columns $A_3$ and $B_3$; in
particular its entries in the columns $A,B,A_2,B_2$ are zero. The rows of the four other edges are, by the
definition of the edge matrices,
\begin{align*}
  Z[A_2,\cdot\,]&=e_A^{\T}, & Z[B_2,\cdot\,]&=(1-x')\,e_A^{\T}+x'e_B^{\T},\\
  Z[A_3,\cdot\,]&=e_{A_2}^{\T}, & Z[B_3,\cdot\,]&=(1-x'')\,e_{A_2}^{\T}+x''e_{B_2}^{\T}
\intertext{in the parallel case, and}
  Z[A_2,\cdot\,]&=e_A^{\T}, & Z[B_3,\cdot\,]&=(1-x')\,e_A^{\T}+x'e_{B_2}^{\T},\\
  Z[A_3,\cdot\,]&=e_{A_2}^{\T}, & Z[B_2,\cdot\,]&=(1-x'')\,e_{A_2}^{\T}+x''e_B^{\T}
\end{align*}
in the antiparallel case. We eliminate the node $A_2$ and then the node $B_2$ by Lemma~\ref{lem.elim}; this is
allowed because $Z[A_2,A_2]=Z[B_2,B_2]=Z[A_2,B_2]=0$. The rows of the edges of $D$ are not changed. In both cases the
rows $A_3$ and $B_3$ become
\begin{equation}\label{eq.bigon}
   e_A^{\T}\qquad\text{and}\qquad(1-X)\,e_A^{\T}+X\,e_B^{\T},\qquad\text{where }X=x'x'' ;
\end{equation}
for instance, in the parallel case the row $B_3$ becomes $(1-x'')e_A^{\T}+x''\bigl((1-x')e_A^{\T}+x'e_B^{\T}\bigr)$. So the
two new crossings act together as a single crossing with the parameter $x'x''$. Let $Z_1$ be the resulting matrix,
indexed by the edges of $D$ and by $A_3,B_3$. By Lemma~\ref{lem.elim}, $P_{D^+,\hate}(x)=\det((1-Z_1)_{\hate})$. The
matrix $Z_1$ is the sum of a matrix which does not involve $x',x''$ and of $X\,e_{B_3}(e_B-e_A)^{\T}$. So, by
Lemma~\ref{lem.rankone}, if $X$ is replaced by an indeterminate, then $\det((1-Z_1)_{\hate})=p_0+X\,p_1$ with
polynomials $p_0,p_1$ in $x_1,\dots,x_c$, and therefore
\[
   P_{D^+,\hate}(x)=p_0+x'x''\,p_1 .
\]
To compute $p_0+p_1$ we put $X=1$. The rows \eqref{eq.bigon} are then $e_A^{\T}$ and $e_B^{\T}$; in particular the
entries of $Z_1$ in the positions $(A_3,A_3)$, $(B_3,B_3)$ and $(A_3,B_3)$ are zero, so that we can eliminate the
node $A_3$ and then the node $B_3$ by Lemma~\ref{lem.elim}. This adds the columns $A_3$ and $B_3$ to the columns
$A$ and $B$ and removes the rows and the columns of $A_3$ and $B_3$, and we recover the matrix $T^a_D(x)$.
Hence $p_0+p_1=P_{D,\hate}(x)$.

\smallskip\emph{The matrix $T^b_{D^+}(y)$.} The same computation applies. The rows of the new edges are now
$y'e_A^{\T}$, $(1-y')e_A^{\T}+e_B^{\T}$, $y''e_{A_2}^{\T}$, $(1-y'')e_{A_2}^{\T}+e_{B_2}^{\T}$ for $A_2,B_2,A_3,B_3$ in the parallel
case, and $y'e_A^{\T}$, $(1-y')e_A^{\T}+e_{B_2}^{\T}$, $y''e_{A_2}^{\T}$, $(1-y'')e_{A_2}^{\T}+e_B^{\T}$ for $A_2,B_3,A_3,B_2$ in the
antiparallel case. After the elimination of $A_2$ and $B_2$ the rows $A_3$ and $B_3$ are, in both cases,
\[
   Y\,e_A^{\T}\qquad\text{and}\qquad(1-Y)\,e_A^{\T}+e_B^{\T},\qquad\text{where }Y=y'y'' .
\]
The resulting matrix depends on $y',y''$ through the term $Y\,(e_{A_3}-e_{B_3})e_A^{\T}$ of rank one, and for $Y=1$ the
rows $A_3$ and $B_3$ are again $e_A^{\T}$ and $e_B^{\T}$, so that the elimination of $A_3$ and $B_3$ gives $T^b_D(y)$. Hence $\tP_{D^+,\hate}(y)=q_0+y'y''\,q_1$ with polynomials $q_0,q_1$ in
$y_1,\dots,y_c$ such that $q_0+q_1=\tP_{D,\hate}(y)$.

\smallskip\emph{The map $\BE$.} Let $m$ be a monomial in $x_1,\dots,x_c,y_1,\dots,y_c$ of degree $\le1$ in each
variable. By the definition \eqref{eq.E} and since $\ve_{c+2}=-\ve_{c+1}$,
\begin{gather*}
   \BE_{D^+}[x'x''m]=t_0^{\ve}t_0^{-\ve}\,\BE_D[m]=\BE_D[m],\qquad \BE_{D^+}[y'y''m]=t_1^{\ve}t_1^{-\ve}\,\BE_D[m]=\BE_D[m],\\
   \BE_{D^+}[x'x''y'y''m]=\BE_D[m].
\end{gather*}
Therefore
$\BE_{D^+}\bigl[(p_0+x'x''p_1)(q_0+y'y''q_1)\bigr]=\BE_D\bigl[(p_0+p_1)(q_0+q_1)\bigr]=\BE_D\bigl[P_{D,\hate}\,\tP_{D,\hate}\bigr]$.

\smallskip\emph{The number of components.} By the combinatorial data of the two new crossings, the edge $A_1$ is
followed by $A_2$ and then by $A_3$ along its strand, and $B_1$ is followed by $B_2$ and then by $B_3$, in the
parallel case and in the antiparallel case. So the cycles of the permutation of Section~\ref{ssec.diagdef} for
$D^+$ are those for $D$ in which $A$ and $B$ have been replaced by $A_1,A_2,A_3$ and $B_1,B_2,B_3$, and the number
$\mu$ is the same for $D^+$ and for $D$. Hence $\Upsilon(D^+,\hate^+)=\Upsilon(D,\hate)$.
\end{proof}

\begin{remark}\label{rem.RII}
The proof shows why $\BE$ is compatible with the moves of type~II: the two crossings of a bigon act as one
crossing with the parameters $x'x''$ and $y'y''$, this crossing acts trivially when both parameters are equal to
$1$, and $\BE$ sends $x'x''$, $y'y''$ and their product to $1$ because the two crossings have opposite signs. The
lemma does not use the fact that the diagrams are drawn on the sphere. Note that it says nothing about the cut
edges $A_2,B_2,A_3,B_3$ of $D^+$; these are covered by Theorem~\ref{thm.diagrams} below.
\end{remark}

\subsection{The edge formula for an arbitrary diagram}\label{ssec.thmdiagrams}
The second result we use is the following form of Alexander's theorem.

\begin{theorem}[Yamada \cite{Yamada}, Vogel \cite{Vogel}]\label{thm.vogel}
Let $D$ be a diagram. There is a finite sequence $D=D_0,D_1,\dots,D_m$ of diagrams of the same link such that each
$D_{i+1}$ is obtained from $D_i$ by a Reidemeister move of type~II between two distinct edges of $D_i$, hence by a
move of type~II in the sense of Section~\ref{ssec.RII}, and such that $D_m$ is, up to an orientation-preserving
homeomorphism of $S^2$, the closure diagram $D_w$ of a braid word $w$.
\end{theorem}

Let us recall the terms, and explain how this statement is contained in \cite{Vogel}; see also \cite{Traczyk} and
\cite[Section~2]{BirmanBrendle}. Smoothing all the crossings of a diagram $D$ in the way which is compatible with
the orientation gives a family of disjoint oriented simple closed curves on $S^2$, the \emph{Seifert circles} of
$D$. Every edge of $D$ lies on one Seifert circle. Near a crossing, the two smoothed arcs run side by side in the
same direction, and the region between them lies to the right of one of them and to the left of the other one;
since each of the two components of the complement of a simple closed curve lies on one side of it, the two arcs
belong to two distinct Seifert circles, and we say that the crossing \emph{joins} these two circles. Two disjoint
oriented simple closed curves on $S^2$ are \emph{coherently oriented} if the annulus between them lies to the left
of one of them and to the right of the other one, and \emph{incoherently oriented} otherwise.

\smallskip\emph{The algorithm.} Vogel's algorithm \cite{Vogel}, which refines the one of Yamada \cite{Yamada}, is the
following. As long as two of the Seifert circles of the diagram are incoherently oriented, there is a face of the
diagram whose boundary contains two edges $A$ and $B$ which lie on two distinct, incoherently oriented Seifert
circles, and one pushes a piece of $A$ over a piece of $B$ across this face. This is a Reidemeister move of
type~II, and a move of type~II in the sense of Section~\ref{ssec.RII}: here $A\ne B$ because the two edges lie on
distinct Seifert circles. (It belongs to the antiparallel case: near $A$ and near $B$ the face lies in the annulus
between the two circles, and this annulus is on the same side of both of them, so the face lies on the same side of
$A$ and of $B$.) Such a move does not change the number of Seifert circles, and it decreases the number of pairs of
incoherently oriented Seifert circles by one \cite{Vogel,Traczyk}. So the procedure stops, at a diagram $D_m$ whose
Seifert circles are pairwise coherently oriented. A move of type~II between two edges of a connected diagram gives a connected diagram,
so all the $D_i$ are diagrams in the sense of Section~\ref{ssec.diagdef}.

\smallskip\emph{The final diagram.} It remains to recall why $D_m$ is a closure diagram, and to compare the
conventions. Let $\Omega$ be a component of the complement of the Seifert circles of $D_m$. If three Seifert circles
lie on the boundary of $\Omega$, then two of them have $\Omega$ on the same side; the annulus between these two
contains $\Omega$, so they are incoherently oriented. Hence the boundary of every component consists of at most two
circles. Now the components and the circles are the vertices and the edges of a tree, because every circle
separates the sphere, and a tree whose vertices have at most two neighbours is a path. So the Seifert circles are
nested: they can be numbered $C_1,\dots,C_s$ in such a way that their complement consists of a disc bounded by $C_1$,
of the annuli between $C_i$ and $C_{i+1}$ for $1\le i<s$, and of a disc bounded by $C_s$. Suppose that the first disc
lies to the left of $C_1$. Then the annulus between $C_1$ and $C_2$ lies to the right of $C_1$, hence to the left of
$C_2$, because these two circles are coherently oriented; therefore the annulus between $C_2$ and $C_3$ lies to the
right of $C_2$, hence to the left of $C_3$, and so on. Thus the annulus between $C_i$ and $C_{i+1}$ lies to the right
of $C_i$ and to the left of $C_{i+1}$ for every $i$, and the last disc lies to the right of $C_s$. In the same way,
if the first disc lies to the right of $C_1$, then the last disc lies to the left of $C_s$, and we number the
circles in the opposite order. So we may and do assume that the first disc lies to the left of $C_1$. A crossing of
$D_m$ joins two distinct Seifert circles across a component of their complement, so it joins $C_i$ and $C_{i+1}$
for some $i$, across the annulus between them.

Now remove a point from the disc bounded by $C_1$ and identify the rest of $S^2$ with the plane by an
orientation-preserving homeomorphism. In the plane, $C_1,\dots,C_s$ are nested closed curves which all run clockwise,
and $C_1$ is the outermost one. This is the picture of the closure diagram of a braid on $s$ strands, drawn with
its strands running upwards and closed on the right, in which $C_i$ is the Seifert circle of the strand in the
position $i$. Thus an orientation-preserving homeomorphism maps $D_m$ to the closure diagram $D_w$ of a braid word
$w$ on $s$ strands: the letters of $w$ are the crossings of $D_m$, and a crossing which joins $C_i$ and $C_{i+1}$
becomes a crossing between the strands in the positions $i$ and $i+1$. (After a further isotopy of the plane the
circles are concentric and the crossings lie at distinct angles; the letters of $w$ are then read clockwise, that
is, in the direction of the circles, starting anywhere.) An orientation-preserving homeomorphism preserves the
over-strand, the under-strand and the sign of each crossing. Since the letter $\sigma_i$ is the
positive crossing and $\sigma_i^{-1}$ the negative one, the combinatorial data of $D_m$ are those of $D_w$, which are
given by \eqref{eq.ou}. Finally $D_m$ is connected, so every position is moved by some letter of $w$.

\begin{theorem}\label{thm.diagrams}
Let $L$ be an oriented link with $\mu$ components, let $D$ be a connected diagram of $L$ with at least one crossing,
and let $\hate$ be an edge of $D$. Then
\[
   V_1(L;t_0,t_1)=(-1)^{\mu-1}\ \BE_D\Bigl[\det\bigl((1-T^a_D(x))_{\hate}\bigr)\,\det\bigl((1-T^b_D(y))_{\hate}\bigr)\Bigr],
\]
where $\BE_D$ is the linear map with $\BE_D[x_j]=t_0^{\ve_j}$, $\BE_D[y_j]=t_1^{\ve_j}$ and $\BE_D[x_jy_j]=1$, multiplicative
over distinct crossings. In the notation \eqref{eq.Upsilon}: $V_1(L)=\Upsilon(D,\hate)$.
\end{theorem}
\begin{proof}
Let $D=D_0,\dots,D_m$ be as in Theorem~\ref{thm.vogel}. Put $\hate_0=\hate$ and let $\hate_{i+1}=(\hate_i)^+$ be the edge
of $D_{i+1}$ associated with the edge $\hate_i$ of $D_i$ in Lemma~\ref{lem.RII}. By this lemma,
$\Upsilon(D_{i+1},\hate_{i+1})=\Upsilon(D_i,\hate_i)$ for all $i$, hence $\Upsilon(D,\hate)=\Upsilon(D_m,\hate_m)$. The diagram
$D_m$ has the same combinatorial data as the closure diagram $D_w$ of a braid word $w$ in which every position is
moved, and $\Upsilon$ depends only on the combinatorial data. All the $D_i$ are diagrams of $L$, so $L$ is the closure
of $w$, and $\Upsilon(D_m,\hate_m)=V_1(L)$ by Proposition~\ref{prop.braidp}.
\end{proof}

\begin{corollary}\label{cor.treesD}
Let $D$ be a connected diagram of $L$ with at least one crossing and let $\hate$ be an edge of $D$. Then the
statements of Proposition~\ref{prop.trees} hold for $D$ and $\hate$: with the $a$-trees and the $b$-trees of
Definition~\ref{def.trees} for the diagram $D$,
\begin{gather*}
   P_{D,\hate}(x)=\sum_{\phi}\ \prod_{j=1}^cp_{\vartheta_j(\phi)}(x_j),\qquad
   \tP_{D,\hate}(y)=\sum_{\psi}\ \prod_{j=1}^cp_{\vartheta_j(\psi)}(y_j),\\
   V_1(L;t_0,t_1)=(-1)^{\mu-1}\sum_{(\phi,\psi)}\ \prod_{j=1}^cW_j\bigl(\vartheta_j(\phi),\vartheta_j(\psi)\bigr).
\end{gather*}
\end{corollary}
\begin{proof}
The proof of Proposition~\ref{prop.trees}(i) only uses Lemma~\ref{lem.mtt} and the fact that every edge is the
edge $o'_j$ or $u'_j$ of exactly one crossing and the edge $o_j$ or $u_j$ of exactly one crossing; it applies to
$T^a_D$ and $T^b_D$. The state sum follows from it as in the proof of Proposition~\ref{prop.trees} (ii), with
Theorem~\ref{thm.diagrams} in place of Theorem~\ref{thm.B}.
\end{proof}

\begin{remark}\label{rem.diagrams}
\begin{enumerate}\renewcommand{\theenumi}{\roman{enumi}}
\item The right-hand side of Theorem~\ref{thm.diagrams} does not depend on the edge $\hate$, but the two determinants
      do. We do not know a direct proof of this independence: here it comes from Theorem~\ref{thm.GHKSTp}, that
      is, from the properties (P$_1$) and (P$_2$) of the $R$-matrix.
\item \emph{Convention.} As in Section~\ref{sec.setup}, the letter $\sigma_i$, which is represented by $R$, is a positive
      crossing; this agrees with \cite{GHKST}, where $R$ is the weight of the positive crossings. The signs $\ve_j$ and
      the over -- and under -- strands in the definition of $T^a_D$ and $T^b_D$ refer to this convention. With the opposite
      convention the link $L$ would have to be replaced by its mirror image in Theorem~\ref{thm.diagrams}. The
      statement of Theorem~\ref{thm.alternating} below is not sensitive to this, because the mirror image of an
      alternating diagram is alternating.
\end{enumerate}
\end{remark}

\section{Alternating links}\label{sec.alternating}

\subsection{Statements}\label{ssec.altstatement}
In this section we combine Theorem~\ref{thm.diagrams} with an argument of Crowell
\cite{Crowell} for the Alexander polynomial, and we obtain that the $V_1$-polynomial of an alternating link is
alternating, proving Theorem~\ref{thm.alternating}. It follows from a formula
without cancellations, which we now state.

An oriented link diagram is \emph{alternating} if each of its components passes alternately over and under at the
crossings which it meets; for a diagram in the sense of Section~\ref{ssec.diagdef} this means that every edge is the
over-strand at exactly one of its two ends. A link is \emph{alternating} if it has an alternating diagram. We will assume alternating diagrams are \emph{reduced}, i.e. there do not exist any crossings which can be removed by applying a Reidemeister I move. 

Let $D$ be a connected alternating diagram with $c\ge1$ crossings. Every edge of $D$ is the over-strand at exactly
one of its ends, so the edges are of two types. Let $\CE_{\rm I}$ be the set of the edges which leave a crossing as
its over-strand and enter the next crossing as its under-strand, and let $\CE_{\rm II}$ be the set of the edges
which leave a crossing as its under-strand and enter the next crossing as its over-strand. Thus
\[
   \CE_{\rm I}=\{o'_1,\dots,o'_c\}=\{u_1,\dots,u_c\},\qquad \CE_{\rm II}=\{u'_1,\dots,u'_c\}=\{o_1,\dots,o_c\},
\]
and each of the two sets has $c$ elements. We fix an edge $\hate\in\CE_{\rm I}$, and we denote by $k_0$ the crossing
which $\hate$ enters and by $j_1$ the crossing which it leaves: $\hate=u_{k_0}=o'_{j_1}$.

\begin{definition}\label{def.alttrees}
An \emph{under-tree} of $(D,\hate)$ is a map $\tau$ which assigns to every crossing $j$ one of the two edges of the
under-strand at $j$, that is $\tau(j)\in\{u_j,u'_j\}$, with the following property. Let $f_\tau(j)=*$ if $\tau(j)=\hate$, and
otherwise let $f_\tau(j)$ be the crossing at the other end of the edge $\tau(j)$: the crossing which $\tau(j)$ leaves if
$\tau(j)=u_j$, and the crossing which it enters if $\tau(j)=u'_j$. Then every crossing is mapped to $*$ by some iterate of
$f_\tau$. We put $S(\tau)=\{j:\tau(j)=u_j\}$.

An \emph{over-tree} of $(D,\hate)$ is a map $\tau'$ which assigns to every crossing $j$ one of the two edges of the
over-strand at $j$, that is $\tau'(j)\in\{o_j,o'_j\}$, with the same property for the map $f_{\tau'}$ defined in the same
way: $f_{\tau'}(j)=*$ if $\tau'(j)=\hate$, and otherwise $f_{\tau'}(j)$ is the crossing at the other end of $\tau'(j)$. We put
$S'(\tau')=\{j:\tau'(j)=o'_j\}$.
\end{definition}

Expressed plainly: in an under-tree one leaves each crossing along its under-strand, forwards or backwards, and from every
crossing one must eventually arrive at the cut edge $\hate$; in an over-tree one leaves each crossing along its
over-strand. Here is an equivalent description. Direct every edge of $D$ from the crossing at which it is the
under-strand to the crossing at which it is the over-strand, and let the edge $\hate$ lead from $k_0$ to a new vertex
$*$ instead of $j_1$. In the directed graph obtained in this way, which may have loops and multiple arcs, the
under-trees are the spanning trees (a tree being a set of edges of $D$) which contain exactly one arc starting at
each crossing; every crossing is then joined to $*$ by a directed path of the tree. The over-trees are the spanning
trees of the same kind when all the edges are directed in the opposite way, where $\hate$ now leads from $j_1$ to
$*$. The edge $\hate$ is an
edge of the under-strand only at $k_0$ and an edge of the over-strand only at $j_1$, and the vertex $*$ can only be
reached along $\hate$; so $\tau(k_0)=\hate$ for every under-tree and $\tau'(j_1)=\hate$ for every over-tree. The sets
$S(\tau)$ and $S'(\tau')$ record the crossings which are left along an edge of $\CE_{\rm I}$.

\begin{theorem}\label{thm.positive}
Let $L$ be a link with $\mu$ components, let $D$ be a connected alternating diagram of $L$ with at least one
crossing, and let $\hate$ be an edge of $D$ which leaves a crossing as its over-strand. Then
\[
   (-1)^{\mu-1}\,V_1(L;-t_0,-t_1)\;=\;\sum_{(\tau,\tau')}\ \ \prod_{j\in S(\tau)\smallsetminus S'(\tau')}t_0^{\ve_j}
        \prod_{j\in S'(\tau')\smallsetminus S(\tau)}t_1^{\ve_j},
\]
where the sum runs over all pairs of an under-tree $\tau$ and an over-tree $\tau'$ of $(D,\hate)$. There is at least one
such pair. Thus the coefficient of $t_0^it_1^k$ in $(-1)^{\mu-1}V_1(L;-t_0,-t_1)$ is the number of pairs $(\tau,\tau')$ such
that $\sum_{j\in S(\tau)\smallsetminus S'(\tau')}\ve_j=i$ and $\sum_{j\in S'(\tau')\smallsetminus S(\tau)}\ve_j=k$.
\end{theorem}

\begin{example}[trefoil]\label{ex.alttrefoil}
Let $D$ be the closure diagram of $\sigma_1^3$, with the notation of Example~\ref{ex.trees}: the edges are the nodes
$(p,j)$ with $p\in\{1,2\}$, $j\in\BZ/3$, and $(o_j,u_j,o'_j,u'_j)=((1,j),(2,j),(2,j+1),(1,j+1))$. The diagram is alternating,
$\CE_{\rm I}$ consists of the edges $(2,j)$ and $\CE_{\rm II}$ of the edges $(1,j)$. Take $\hate=(2,1)=u_1=o'_3$, so that
$k_0=1$ and $j_1=3$. For an under-tree, $f_\tau(j)$ is $j+1$ if $\tau(j)=u'_j$ and $j-1$ if $\tau(j)=u_j$, except that
$f_\tau(1)=*$ if $\tau(1)=u_1$. The root $*$ can only be reached from the crossing $1$, so $\tau(1)=u_1$; of the four choices
at the crossings $2$ and $3$ only $\tau(2)=u'_2$, $\tau(3)=u_3$ is excluded, because it gives the cycle $2\mapsto3\mapsto2$.
So there are three under-trees, with $S(\tau)=\{1\}$, $\{1,2\}$, $\{1,2,3\}$. Likewise $f_{\tau'}(j)$ is $j-1$ if
$\tau'(j)=o_j$ and $j+1$ if $\tau'(j)=o'_j$, except that $f_{\tau'}(3)=*$ if $\tau'(3)=o'_3$, and there are three over-trees, with
$S'(\tau')=\{3\}$, $\{2,3\}$, $\{1,2,3\}$. All the crossings are positive, and the nine pairs contribute
\[
  t_0t_1+t_0t_1^2+t_1^2\;+\;t_0^2t_1+t_0t_1+t_1\;+\;t_0^2+t_0+1
     \;=\;V_1(3_1;-t_0,-t_1),
\]
in agreement with Example~\ref{ex.trefoil}.
\end{example}

\subsection{The two determinants of an alternating diagram}\label{ssec.crowell}
We keep the notation of Section~\ref{ssec.altstatement}. Every edge of $\CE_{\rm II}$ is the edge $u'_j$ of exactly one
crossing and the edge $o_k$ of exactly one crossing, and every edge of $\CE_{\rm I}$ is the edge $o'_j$ of exactly one
crossing and the edge $u_k$ of exactly one crossing. So there are two permutations $\mathsf n$ and $\mathsf m$ of
$\{1,\dots,c\}$ such that
\[
    u'_j=o_{\mathsf n(j)}\qquad\text{and}\qquad o'_j=u_{\mathsf m(j)}\qquad\text{for all }j .
\]
Thus $\mathsf n(j)$ is the crossing which the edge $u'_j$ enters, $\mathsf m(j)$ is the crossing which the edge $o'_j$
enters, and $k_0=\mathsf m(j_1)$. Nothing is assumed about these permutations: we may have $\mathsf n(j)=j$ or
$\mathsf m(j)=j$ (the edge is then a loop), $\mathsf n(j)=\mathsf m^{-1}(j)$ (the two edges of the under-strand at $j$ lead
to the same crossing), or $k_0=j_1$.
The following proposition is the argument of Crowell \cite{Crowell}, with one variable for each crossing.

\begin{proposition}\label{prop.crowell}
Let $D$ be a connected alternating diagram with $c\ge1$ crossings and let $\hate\in\CE_{\rm I}$. Then
\[
   P_{D,\hate}(x)=(-1)^c\sgn(\mathsf n)\sum_{\tau}\ \prod_{j\in S(\tau)}(-x_j),\qquad
   \tP_{D,\hate}(y)=(-1)^c\sgn(\mathsf n)\sum_{\tau'}\ \prod_{j\in S'(\tau')}(-y_j),
\]
where $\tau$ runs over the under-trees and $\tau'$ over the over-trees of $(D,\hate)$. Moreover $(D,\hate)$ has at least
one under-tree and at least one over-tree.
\end{proposition}
\begin{proof}
\emph{Step 1: elimination of the edges of $\CE_{\rm I}$.} Let $Z$ be the matrix $T^a_D(x)$ with the row and the column
of $\hate$ removed, so that $P_{D,\hate}(x)=\det(1-Z)$. By the definition of $T^a_D$, the rows of $Z$ are
\[
   Z[o'_j,\cdot\,]=e_{o_j}^{\T}\quad(o'_j\ne\hate),\qquad\qquad Z[u'_j,\cdot\,]=(1-x_j)\,e_{o_j}^{\T}+x_j\,e_{u_j}^{\T},
\]
where the term $x_je_{u_j}^{\T}$ is absent if $u_j=\hate$, that is, if $j=k_0$. The rows of the edges of $\CE_{\rm I}$ have
their only non-zero entry in a column of $\CE_{\rm II}$. We eliminate the nodes of $\CE_{\rm I}\smallsetminus\{\hate\}$ one
after the other by Lemma~\ref{lem.elim}. The elimination of $\zeta\in\CE_{\rm I}$ adds $Z[k,\zeta]Z[\zeta,l]$ to the entry
$Z[k,l]$; this is non-zero only if $k,l\in\CE_{\rm II}$, so the rows of the remaining edges of $\CE_{\rm I}$ are not
changed and their diagonal entries remain zero. Since $u_j=o'_{\mathsf m^{-1}(j)}$, we end up with the matrix
$Z_{\rm II}$ indexed by $\CE_{\rm II}$ whose rows are
\[
   Z_{\rm II}[u'_j,\cdot\,]=(1-x_j)\,e_{o_j}^{\T}+x_j\,e_{o_{\mathsf m^{-1}(j)}}^{\T}\qquad(\text{the last term is absent if }j=k_0),
\]
and $P_{D,\hate}(x)=\det(1-Z_{\rm II})$.

\smallskip\emph{Step 2: a Laplacian.} We index the rows and the columns of $1-Z_{\rm II}$ by the crossings, the
crossing $l$ standing for the edge $o_l$. The row of the edge $u'_j=o_{\mathsf n(j)}$ is then the row $\mathsf n(j)$, and
it is $e_{\mathsf n(j)}^{\T}-(1-x_j)e_j^{\T}-x_je_{\mathsf m^{-1}(j)}^{\T}$. Let $\Lambda$ be the $c\times c$ matrix whose row $j$ is
minus the row $\mathsf n(j)$ of $1-Z_{\rm II}$:
\begin{equation}\label{eq.Lambda}
   \Lambda[j,\cdot\,]=(1-x_j)\,e_j^{\T}-e_{\mathsf n(j)}^{\T}+x_j\,e_{\mathsf m^{-1}(j)}^{\T}\qquad(\text{the last term is absent if }j=k_0).
\end{equation}
Then $\Lambda=-\mathsf N\,(1-Z_{\rm II})$, where $\mathsf N$ is the permutation matrix with $\mathsf N[j,\mathsf n(j)]=1$ for
all $j$, and therefore $P_{D,\hate}(x)=\det(1-Z_{\rm II})=(-1)^c\sgn(\mathsf n)\det\Lambda$.

\smallskip\emph{Step 3: the matrix-tree theorem.} By Definition~\ref{def.alttrees}, the crossing at the other end of
$u'_j$ is $\mathsf n(j)$ and the crossing at the other end of $u_j$ is $\mathsf m^{-1}(j)$. We give the weight $1$ to the
edge $u'_j$ and the weight $-x_j$ to the edge $u_j$. For a crossing $j$ and $l\in\{1,\dots,c,*\}$ let $\mathsf w_j(l)$ be
the sum of the weights of the edges $e\in\{u'_j,u_j\}$ which lead from $j$ to $l$; here $u'_j$ leads to $\mathsf n(j)$,
and $u_j$ leads to $\mathsf m^{-1}(j)$ if $j\ne k_0$ and to $*$ if $j=k_0$. Thus $\sum_l\mathsf w_j(l)=1-x_j$, and by
\eqref{eq.Lambda} $\Lambda[j,l]=(1-x_j)\d_{jl}-\mathsf w_j(l)$ for $l\ne*$; this also holds when two of the crossings $j$,
$\mathsf n(j)$, $\mathsf m^{-1}(j)$ coincide. We work over the field of rational functions in $x_1,\dots,x_c$. Let $Z^*$ be
the matrix indexed by $\{1,\dots,c,*\}$ with $Z^*[j,l]=\mathsf w_j(l)/(1-x_j)$ and with the row of $*$ equal to zero. Its
rows other than the row of $*$ have sum $1$, and $(1-Z^*)_{\hat *}$ is the matrix $\Lambda$ with its row $j$ divided by
$1-x_j$, for every $j$. By Lemma~\ref{lem.mtt},
\[
   \det\Lambda=\prod_{j=1}^c(1-x_j)\,\det\bigl((1-Z^*)_{\hat *}\bigr)=\sum_f\ \prod_{j=1}^c\mathsf w_j(f(j)),
\]
where $f$ runs over the trees rooted at $*$. Expanding each $\mathsf w_j(f(j))$ as a sum over the edges
$\tau(j)\in\{u'_j,u_j\}$ which lead from $j$ to $f(j)$, we obtain the sum over all maps $\tau$ such that $f_\tau$ is a tree
rooted at $*$, that is, over the under-trees, of the products $\prod_{j\in S(\tau)}(-x_j)$. This proves the formula for
$P_{D,\hate}$.

\smallskip\emph{Step 4: the polynomial $\tP_{D,\hate}$.} A matrix and its transpose have the same determinant, so
$\tP_{D,\hate}(y)=\det(1-Z)$, where now $Z$ is the transpose of the matrix $T^b_D(y)$ with the row and the column of
$\hate$ removed. By the definition of $T^b_D$, the rows of $Z$ are
\[
   Z[u_j,\cdot\,]=e_{u'_j}^{\T}\quad(u_j\ne\hate),\qquad\qquad Z[o_j,\cdot\,]=(1-y_j)\,e_{u'_j}^{\T}+y_j\,e_{o'_j}^{\T},
\]
where the term $y_je_{o'_j}^{\T}$ is absent if $o'_j=\hate$, that is, if $j=j_1$. The situation is the same as in Step~1,
with the roles of the over-strands and of the under-strands exchanged: the rows of the edges of $\CE_{\rm I}$ have
their only non-zero entry in a column of $\CE_{\rm II}$. Eliminating the edges of $\CE_{\rm I}\smallsetminus\{\hate\}$, and
using $o'_j=u_{\mathsf m(j)}$, we obtain the matrix indexed by $\CE_{\rm II}$ with the rows
$(1-y_j)e_{u'_j}^{\T}+y_je_{u'_{\mathsf m(j)}}^{\T}$ for the edges $o_j$, where the last term is absent if $j=j_1$. We now let
the crossing $l$ stand for the edge $u'_l$; the edge $o_j=u'_{\mathsf n^{-1}(j)}$ is then the row $\mathsf n^{-1}(j)$. With
\[
   \Lambda'[j,\cdot\,]=(1-y_j)\,e_j^{\T}-e_{\mathsf n^{-1}(j)}^{\T}+y_j\,e_{\mathsf m(j)}^{\T}\qquad(\text{the last term is absent if }j=j_1)
\]
we get, as in Step~2, $\tP_{D,\hate}(y)=(-1)^c\sgn(\mathsf n^{-1})\det\Lambda'=(-1)^c\sgn(\mathsf n)\det\Lambda'$. The crossing
at the other end of $o_j$ is $\mathsf n^{-1}(j)$ and the crossing at the other end of $o'_j$ is $\mathsf m(j)$, so Step~3,
with the weight $1$ for $o_j$ and the weight $-y_j$ for $o'_j$, gives
$\det\Lambda'=\sum_{\tau'}\prod_{j\in S'(\tau')}(-y_j)$.

\smallskip\emph{Step 5: existence of trees.} Direct every edge of $D$ from the crossing at which it is the
under-strand to the crossing at which it is the over-strand. In the directed graph obtained in this way, whose
vertices are the crossings, two arcs start at every vertex (the edges $u_j,u'_j$) and two arcs end at every vertex
(the edges $o_j,o'_j$). Let $\mathcal R$ be the set of the crossings from which there is a directed path to $k_0$.
No arc leads from the complement of $\mathcal R$ to $\mathcal R$. So the $2\,|\mathcal R^c|$ arcs which start in
$\mathcal R^c$ also end in $\mathcal R^c$; since exactly $2\,|\mathcal R^c|$ arcs end in $\mathcal R^c$, no arc leads
from $\mathcal R$ to $\mathcal R^c$ either. As $D$ is connected and $k_0\in\mathcal R$, every crossing belongs to
$\mathcal R$. Now let $\tau(k_0)=\hate$, and for $j\ne k_0$ let $\tau(j)$ be the first arc of a directed path of minimal
length from $j$ to $k_0$. Then $f_\tau(j)$ is closer to $k_0$ than $j$ for $j\ne k_0$, and $f_\tau(k_0)=*$, so $\tau$ is an
under-tree. Reversing all the arcs and replacing $k_0$ by $j_1$, we find an over-tree in the same way.
\end{proof}

\subsection{Proofs of Theorems~\ref{thm.positive} and~\ref{thm.alternating}}\label{ssec.altproofs}

\begin{proof}[Proof of Theorem~\ref{thm.positive}]
By Theorem~\ref{thm.diagrams},
\[
   V_1(L;t_0,t_1)=(-1)^{\mu-1}\,\BE_D\bigl[P_{D,\hate}(x)\,\tP_{D,\hate}(y)\bigr].
\]
By Proposition~\ref{prop.crowell}, in which the two signs $(-1)^c\sgn(\mathsf n)$ cancel,
\[
   P_{D,\hate}(x)\,\tP_{D,\hate}(y)=\sum_{(\tau,\tau')}\ \prod_{j\in S(\tau)}(-x_j)\prod_{j\in S'(\tau')}(-y_j).
\]
Let $S=S(\tau)$ and $S'=S'(\tau')$. By the definition \eqref{eq.E} of $\BE_D$, and since $|S|+|S'|$ and
$|S\smallsetminus S'|+|S'\smallsetminus S|$ have the same parity,
\[
   \BE_D\Bigl[\,\prod_{j\in S}(-x_j)\prod_{j\in S'}(-y_j)\Bigr]
     =\prod_{j\in S\smallsetminus S'}\bigl(-t_0^{\ve_j}\bigr)\prod_{j\in S'\smallsetminus S}\bigl(-t_1^{\ve_j}\bigr)
     =\prod_{j\in S\smallsetminus S'}(-t_0)^{\ve_j}\prod_{j\in S'\smallsetminus S}(-t_1)^{\ve_j},
\]
because $-t^{\ve}=(-t)^{\ve}$ for $\ve=\pm1$. Hence
\[
   V_1(L;t_0,t_1)=(-1)^{\mu-1}\sum_{(\tau,\tau')}\ \prod_{j\in S(\tau)\smallsetminus S'(\tau')}(-t_0)^{\ve_j}
        \prod_{j\in S'(\tau')\smallsetminus S(\tau)}(-t_1)^{\ve_j},
\]
and the formula of the theorem follows by replacing $t_0,t_1$ by $-t_0,-t_1$. The set of pairs is not empty by the
last statement of Proposition~\ref{prop.crowell}.
\end{proof}

For links which have a disconnected alternating diagram we need the following well-known fact.

\begin{lemma}\label{lem.split}
If $L$ is a split link, then $V_1(L)=0$.
\end{lemma}
\begin{proof}
Let $L$ be the union of two links $L_1$ and $L_2$ which are separated by a sphere. By Alexander's theorem, $L_1$
and $L_2$ are the closures of braid words $w_1$ and $w_2$ on $n_1$ and $n_2$ strands, and $L$ is the closure of the
word $w$ on $n_1+n_2$ strands obtained by placing $w_2$ to the right of $w_1$; no letter of $w$ is $\sigma_{n_1}^{\pm1}$.
Hence the matrix $B(x)$ of \eqref{eq.B} is block diagonal, $B(x)=B_1\oplus B_2$ with blocks of sizes $n_1$ and $n_2$, and
$1-B(x)'=(1-B_1')\oplus(1-B_2)$. Each of the matrices $\rho(x_j)$, $\rho^-(x_j)$ fixes the vector $(1,1)^{\T}$, so $B_2$ fixes
the vector $(1,\dots,1)^{\T}$ and $\det(1-B_2)=0$. Therefore $P_w=0$, and $V_1(L)=0$ by Theorem~\ref{thm.B}.
\end{proof}

\begin{proof}[Proof of Theorem~\ref{thm.alternating}]
Let $D$ be an alternating diagram of $L$. If $D$ is not connected, then $L$ is a split link and $V_1(L)=0$ by
Lemma~\ref{lem.split}. If $D$ is connected and has no crossing, then $L$ is the unknot and $V_1(L)=1$
(Example~\ref{ex.unknot}). If $D$ is connected and has at least one crossing, then $\CE_{\rm I}$ is not empty, and
Theorem~\ref{thm.positive} shows that $(-1)^{\mu-1}V_1(L;-t_0,-t_1)$ is a non-empty sum of monomials with coefficient
$1$; in particular its coefficients are non-negative and it is not zero. Finally, the coefficient of $t_0^it_1^k$ in
$V_1(L;-t_0,-t_1)$ is $(-1)^{i+k}$ times the coefficient of $t_0^it_1^k$ in $V_1(L;t_0,t_1)$.
\end{proof}

\begin{remark}\label{rem.altcomments}
\begin{enumerate}\renewcommand{\theenumi}{\roman{enumi}}
\item \emph{The number of terms.} For $t_0=t_1=1$, Theorem~\ref{thm.positive} says that $(-1)^{\mu-1}V_1(L;-1,-1)$ is
      the number of under-trees times the number of over-trees. By Proposition~\ref{prop.crowell} these two
      numbers are $|P_{D,\hate}|$ and $|\tP_{D,\hate}|$ at $x_j=y_j=-1$. 
\item \emph{The Alexander polynomial.} For $x_j=t^{\ve_j}$, Proposition~\ref{prop.crowell} gives
      \[
         P_{D,\hate}(t^{\ve})=\pm\sum_\tau\ \prod_{j\in S(\tau)}(-t)^{\ve_j}.
      \]
      Together with the classical fact that $P_{D,\hate}(t^{\ve})$ is
      the Alexander polynomial of $L$ up to a sign and a power of $t$ (compare Remark~\ref{rem.burau}), this is the
      theorem that the Alexander polynomial of an alternating link is alternating, which is due to Crowell
      \cite{Crowell} and, for knots, to Murasugi \cite{Murasugi}. Theorem~\ref{thm.positive} is the same statement for two coupled copies: the map
      $\BE_D$ preserves the signs because it sends $(-x_j)(-y_j)$ to $1$.
\item \emph{Regarding the hypothesis.} The diagram has to be alternating in Proposition~\ref{prop.crowell}. For the closure
      diagram of $w=\sigma_1^3\sigma_2\sigma_1^{-1}\sigma_2$, a non-alternating diagram of the knot $5_2$, the polynomial $P_w$
      contains the monomials $1$ and $-x_1x_2x_3x_5$, which have opposite signs as monomials in the $-x_j$. And for
      the closure of $(\sigma_1\sigma_2)^4$, the non-alternating knot $8_{19}$, the polynomial $V_1$ contains the monomials
      $1$ and $t_1^3$ with the coefficient $+1$, so it is not alternating.
\end{enumerate}
\end{remark}

\section*{Tool Use}No AI tools have been used to write or create this work.

\end{document}